\documentclass[12pt, a4paper,reqno]{amsart}
\usepackage{amsmath,amssymb,amsthm,color, euscript, enumerate, hhline}
\usepackage{url}

\theoremstyle{plain}

\newtheorem{theorem}{Theorem}[section]
\newtheorem{proposition}[theorem]{Proposition}
\newtheorem{lemma}[theorem]{Lemma}

\theoremstyle{definition}
\newtheorem{definition}[theorem]{Definition}

\theoremstyle{remark}

\numberwithin{equation}{section}

\address[T. Kondo] {Graduate School of Science and Technology,
Kumamoto University, Kumamoto 860-8555, Japan} 
  \email{258d9001@st.kumamoto-u.ac.jp}

\title[Lattices from Construction A and B]{Classification of isomorphism classes of lattices from Construction A and B}

\author{Takara\ Kondo} 

\begin{document}
\maketitle
\begin{abstract}
In this paper, we completely classify the isomorphism classes of certain lattices $L_A(C)$ and $L_B(C)$ from a self-orthogonal code $C$ over the finite field $\mathbb{F}_p$, where $p$ is an odd prime. These lattices are obtained by \emph{Construction A} and \emph{B} for a code $C$ over $\mathbb{F}_p$ introduced by Lam and Shimakura, which arose from a study of orbifolds of lattice vertex operator algebras.  
For self-orthogonal codes $C$ and $D$ of the same length over $\mathbb{F}_p$, we show that $L_X(C) \cong L_X(C)$ as lattices if and only if $C \cong D$ as codes, where $X=A$ or $B$.
This can be expected to be lattice analogues of classifications of the isomorphism classes of lattice vertex operator algebras and its orbifolds. 
To prove the result, we generalize the notion of a frame of a lattice and define some codes which are analogues of codes constructed from Kleinian codes studied by H{\"o}hn.

\end{abstract}

\section{Introduction}
The lattice vertex operator algebra (VOA) is an important class of VOAs \cite{FLM}.  
Every even lattice $L$ has the automorphism defined by scalar multiplication $-1$, and it can be uniquely lifted to an automorphism of order two of the lattice VOA $V_L$, up to conjugation.  
In general, for a VOA and a subgroup of the automorphism group, we obtain the fixed-point subalgebra.  
Such a VOA is called an \emph{orbifold}.  
The VOA $V_L^+$ denotes the orbifold of $V_L$ obtained from a lift of $-1$.

In \cite{Shima}, the isomorphism classes of lattice VOAs $V_L$ associated with even lattices $L$ and the orbifolds $V_L^+$ were completely determined. Let $L$ and $N$ be even lattices of the same rank. If  $L \cong N$ as lattices, then we clearly have $V_L \cong V_N$ and $V_L^+ \cong V_N^+$ as VOAs. Hence, the main problem for the classification is to determine the non-isomorphic even lattices of the same rank which give the isomorphic lattice VOAs and orbifolds.
In \cite{Shima}, it was shown that $V_L \cong V_N$ as VOAs if and only if $L \cong N$ as lattices, and  $V_L^+ \cong V_N^+$ as VOAs if and only if $L \cong N$ as lattices or $\{L,N\} =\{E_8^2, D_{16}^+\}$. Here $E_8^2$ and $D_{16}^+$ are the even unimodular lattices of rank $16$. Since there are many analogies among codes, lattices, and VOAs \cite{Ho95, Ho, Ho08, Sh11}, it is natural to consider lattice analogues of the classification. 
Let $L_A(C)$ and $L_B(C)$ denote the lattices constructed by Construction A and B from a binary code $C$ \cite[Chapter 7]{CS}.
In \cite{KKM, Shima}, it was shown that for doubly-even binary codes $C$ and $D$ of the same length, $L_A(C) \cong L_A(D)$ as lattices if and only if $C \cong D$ as codes, and $L_B(C) \cong L_B(D)$ if and only if $C\cong D$ as codes or $\{C,D\}=\{e_8^2, d_{16}^+\}$. Here $e_8^2$ and $d_{16}^+$ are the doubly-even self-dual binary codes of length $16$. 

An automorphism of a lattice is called \emph{fixed-point free} if it fixes only the zero vector.  
Every lattice has the unique fixed-point free automorphism of order two, which is the scalar multiplication $-1$. However, such transformations of odd prime order $p$ do not exist in general.
If an even lattice $L$ has one, then it can be uniquely lifted to an automorphism of order $p$ of the VOA $V_L$, up to conjugation (cf.\ \cite[PROPOSITION 4.6]{LS1}). Through a study of the orbifold of $V_L$, Lam and Shimakura generalized the lattices $L_A(C)$ and $L_B(C)$ for binary codes $C$ to the ones for codes over any finite field $\mathbb{F}_p$ of prime order $p$ \cite{LS}.  
The lattice $L_A(C)$ from a code $C$ over $\mathbb{F}_p$ has natural fixed-point free automorphisms of order $p$, which are those of $L_B(C)$ under a certain condition. 
For an odd prime $p$, we can expect that $L_A(C)$ and $L_B(C)$ from self-orthogonal codes $C$ over $\mathbb{F}_p$ are lattice analogues of the VOA $V_L$ associated with even lattices $L$ and the orbifold $V_L^{\langle \hat{g}\rangle}$ obtained from lifts $\hat{g}$ of fixed-point free automorphisms $g$ of order $p$, respectively.

In this paper, motivated by the above, we completely determine the isomorphism classes of $L_A(C)$ and $L_B(C)$ from self-orthogonal codes over $\mathbb{F}_p$ for any odd prime $p$. Our main result is as follows:

\begin{theorem}\label{main theorem}
Let $p$ be an odd prime,  $C$ and $D$ self-orthogonal codes of length $k$ over $\mathbb{F}_p$. Then 
$L_X(C) \cong L_X(D)$ as lattices if and only if $C \cong D$ as codes, where $X=A$ or $B$.


\end{theorem}

The above theorem for $X=A$ is a lattice analogue of the classification of the isomorphism classes of lattice VOAs and the theorem for $X=B$ can be expected to be a lattice analogue of the classification for the orbifolds obtained from lifts of fixed-point free automorphisms of order $p$.

Let us explain our method to prove the result. Set $p$, $C$, $D$, and $X$ be as in Theorem \ref{main theorem}. If $C \cong D$ as codes, we clearly have $L_X(C)  \cong L_X(D)$ as lattices. Therefore, the main problem in Theorem \ref{main theorem} is to show the converses. 
For this purpose, we generalize the notion of Type A frame in \cite{KKM} by using the root system of type $A_n$, which is called \emph{$A_n$-frame}. 
By the definition of Construction A, the lattice $L_A(C)$ has an $A_{p-1}$-frame.
The automorphism group $\operatorname{Aut}(L_A(C))$ acts on the set of all $A_{p-1}$-frames of $L_A(C)$.
We show that the action is transitive and that the transitivity implies Theorem \ref{main theorem} for $X=A$. In particular, for the case $p=5$, we study the action of the Weyl group on certain subsets of roots of the lattice $E_8$.    
To prove Theorem \ref{main theorem} for $X = B$ in the cases $p=3,5$, we define some codes which are analogues of codes constructed by \text{Construction} A and B from Kleinian codes in \cite[Section 7]{Ho} and \cite[Subsection 1.5]{Shima}. By using properties of the codes, we show that in most cases an isomorphism between $L_B(C)$ and $L_B(D)$ induces one between $L_A(C)$ and $L_A(D)$. Hence, by the theorem for $X=A$, we have the desired code isomorphism.
The remaining cases are verified by using  MAGMA \cite{MAGMA} and direct calculations.

The organization of this paper is as follows.
In Section~2, we review basic notions on codes, lattices, and the definitions of Construction A and B.
We also give lower bounds of norms of vectors related to lattices constructed by Construction A and B.
These bounds are used to restrict possible cases in the proofs of Theorem \ref{main theorem}.
In Section~3, we introduce the notion of $A_n$-frame.
Furthermore, we show that the action of the automorphism group $\operatorname{Aut}(L_A(C))$ is transitive and that the transitivity implies Theorem~\ref{main theorem} for $X=A$. In Section~4, for the cases $p=3,5$, we define some codes which are analogues of codes constructed by Construction A and B from Kleinian codes.
Moreover, we show some properties of the codes to prove Theorem~\ref{main theorem} for $X=B$.

\section{Preliminaries}\label{Preliminaries}
\subsection{Codes}
Let $p$ be a prime and  $\mathbb{F}_p=\{0,1,\ldots, p-1\}$ the finite field with $p$ elements. A subspace of $\mathbb{F}_p^{k}$ is called a \emph{code of length $k$ over $\mathbb{F}_p$}. For a code $C \subset \mathbb{F}_p^{k}$, we define the \emph{dual code} $C^{\perp}$ of $C$ by 
\[C^{\perp} = \{x \in \mathbb{F}_p^{k} \mid (x,c) =0\ \text{for any}\ c \in C\},\] 
where $(x,y)=\sum_{i=1}^{k} x_iy_i$ for $x=(x_1, x_2, \ldots, x_k)$, $y=(y_1, y_2, \ldots, y_k) \in \mathbb{F}_p^{k}$. 
 A code $C \subset \mathbb{F}_p^{k}$ is called \emph{self-orthogonal} if $C \subset C^{\perp}$.  For a non-negative integer $n$ and a code $C \subset \mathbb{F}_p^{k}$, the subset $C(n)$ of $C$ is defined as follows: 
\[C(n)= \{c \in C \mid \ \text{the number of  non-zero coordinates of $c$ is $n$}\}.\]  
For $c=(c_1, \ldots, c_k) \in \mathbb{F}_p^{k}$, $\operatorname{Supp}c$ denotes the set $\{i \mid c_i \neq 0 \}$.
Let $C$ and $D$ be codes of length $k$ over $\mathbb{F}_p$. Then $C$ and $D$ are called \emph{isomorphic} if there exists $f \in \langle \sigma_i, \sigma \mid 1\le i \le k, \sigma \in S_k \rangle$ such that $f(C)=D$, where  $\sigma_i$ acts on $\mathbb{F}_p^k$ as $-1$ on the $i$-th coordinate and as $1$ on the other coordinates, and each element of the symmetric group $S_k$ induces a permutation of the coordinates.

\subsection{Lattices}\label{lattices}
A \emph{lattice}  is a free abelian group of finite rank with a rational valued, positive-definite symmetric bilinear form $(\cdot ,\cdot)$.  Let $L$ be a lattice. The \emph{rank of a lattice} $L$ is the  rank as a free abelian group, which is denoted by $\operatorname{rank}L$. 
The \emph{dual lattice} $L^*$ of $L$ is defined by 
\[L^*=\{ x \in \mathbb{Q}\otimes_{\mathbb{Z}} L\mid (x,y) \in \mathbb{Z}\ \text{for any}\ y \in L \}.\]
Let $\operatorname{Aut}(L)$ be the automorphism group of the lattice $L$. 
For $S \subset L$, $N_k(S)$ denotes the squared norm of the $k$-th shortest vectors of $S \setminus{\{\boldsymbol{0}\}}$. For $r \in \mathbb{Q}$, the subset $L(r)$ of $L$ is defined as follows:
\[L(r) = \{x \in L \mid (x,x)=r\}.\] 
 A lattice $L$ is called \emph{even} if $(x,x) \in 2\mathbb{Z}$ for any $x \in L$. Note that $L(2)$ forms a root system  if $L$ is even.

\subsection{Definitions of Construction A and B}\label{Def of Const A and B}
We recall the definitions of Construction A and B, and their properties \cite[Section 4]{LS}. We define a lattice $A_{n-1}$ as follows:
\[A_{n-1} = \{(x_1, x_2, \ldots , x_{n}) \in \mathbb{Z}^{n} \mid x_1+x_2+\cdots +x_{n}=0\}.\]
Let $\{e_1, e_2, \ldots, e_n \}$ be the canonical basis of $\mathbb{R}^n$ and set $\alpha_i =e_i-e_{i+1}$ for $1 \le i \le n-1$. Note that $\Delta =\{\alpha_1, \alpha_2, \ldots, \alpha_{n-1}\}$ is a base of the root system of the lattice $A_{n-1}$. As  \cite[Proof of Proposition 4.2.3]{Martinet}, we define the vectors $\varepsilon_i\ (1 \le i \le n)$ as follows:
\[\displaystyle \varepsilon_i =-\frac{1}{n}(e_1+e_2+ \cdots+e_{i-1}-(n-1)e_i+e_{i+1}+ \cdots+ e_n).\]
Note that $\{\varepsilon_1, \varepsilon_2, \ldots, \varepsilon_{n-1}\}$ is a basis of the dual lattice $A_{n-1}^*$ and
\[A_{n-1}^{*}/A_{n-1} = \langle \varepsilon_1+A_{n-1} \rangle \cong \mathbb{Z}_n.\] 
Moreover, we have
\begin{equation}\label{varepsilon:inner product} 
(\varepsilon_i, \varepsilon_i)=\frac{n-1}{n}\ \text{and}\ (\varepsilon_i, \varepsilon_j)=-\frac{1}{n}\ \text{for}\ i \ne j,
\end{equation}
where $(\cdot, \cdot)$ denotes the standard inner product of $\mathbb{R}^n$.

Let $p$ be a prime and  $C$ a self-orthogonal code of length $k$ over $\mathbb{F}_p$. Then we can identify $C \subset \mathbb{Z}_p^k$ with a subset of $(A^{*}_{p-1}/A_{p-1})^{k} = \langle \varepsilon_1+A_{p-1} \rangle^k \cong \mathbb{Z}_p^k$. Let $\pi : (A^{*}_{p-1})^{k} \rightarrow (A^{*}_{p-1}/A_{p-1})^{k}$ be the canonical map.
Then $L_A(C)  = \pi^{-1}(C)$ is called the lattice constructed by \emph{Construction A} from $C$. 

 Let $\Delta_{+}$ be the set of all positive roots with respect to the base $\Delta$. Set $\rho_{\Delta} = (1/2)\sum_{\alpha \in \Delta_{+}} \alpha$ and $\chi_{\Delta} =(1/p)(\rho_{\Delta}, \ldots, \rho_{\Delta}) \in \bigoplus_{i=1}^{k} \mathbb{Q}^{p}$. Then $L_B(C)$ called the lattice constructed by \emph{Construction B} from $C$ is defined as follows:
\[L_B(C) = \{x \in L_A(C) \mid (x,\chi_{\Delta}) \in \mathbb{Z}\}.\]

In this subsection, we now assume that $p$ is an odd prime.
\begin{proposition}{\rm (\cite[Proposition 4.12]{LS})}
The lattice $L_A(C)$ is even if and only if $C$ is self-orthogonal.
\end{proposition}

By the above proposition, $L_A(C)(2)$ forms a root system if $C$ is self-orthogonal.
\begin{lemma}\label{Structures of the dual lattices of Construction B}{\rm (\cite[Lemma 4.14]{LS})}
$L_B(C)^{*}=L_A(C^{\perp})+\mathbb{Z}\chi_{\Delta}.$
\end{lemma}

Set $\lambda_0=0$, $\lambda_1 =\varepsilon_1$, and $\lambda_j=j\varepsilon_1-\sum_{i=1}^{j-1}(j-i)\alpha_i$ for $2 \le j \le p-1$.
For $c=(c_1, \ldots, c_{k}) \in C$, we define $\lambda_{c}$ as follows:
\begin{equation}\label{def of lambda}
\lambda_{c}=(\lambda_{c_1},\ldots, \lambda_{c_k}).
\end{equation} 
\begin{lemma}{\rm (\cite[Lemmas 4.5 and 4.15]{LS})}\label{generating set: ConstA and B} 
Let $X$ be a generating set of $C$. Then the following hold:
\begin{enumerate}
\item The lattice $L_A(C)$ is generated by the sets  $\{\lambda_c \mid c \in X\}$ and $A_{p-1}^k$. 
\item The lattice $L_B(C)$ is generated by the sets $\{\lambda_c \mid c \in X\}$ and $L_B(\{\boldsymbol{0}\})$.  
\end{enumerate}
\end{lemma}

\subsection{Properties of the dual lattice $A^{*}_{n-1}$ and related topics}
In this subsection, we discuss some properties of the lattice $A_{n-1}^*$ and related topics.
For a lattice $L$ and a subset $S$ of $L$, let $N_k(S)$ be as in  Subsection \ref{lattices}.
\begin{lemma}\label{minimal norm}{\rm (\cite[Proposition 4.2.3]{Martinet})}
$N_1(A_{n-1}^{*})=(n-1)/n.$
\end{lemma}

By using induction, we have the following lemma.

\begin{lemma}\label{lem}
Let $n \ge 2$. If $x_i \in \mathbb{Z}$ for $1 \le i \le n$ and there exist $i$,$j$ such that $x_i \ne x_j$, then we have
\[\sum_{1\le i < j \le n} (x_i-x_j)^2 \ge n-1.\]

\end{lemma} 

By using the above lemma, we have a lower bound of $N_2(A_{n-1}^{*})$ for $n \ge 7$.
\begin{lemma}\label{lower bound}
$N_2(A_{n-1}^{*}) > (n+1)/n $ for $n \ge 7$
\end{lemma}
\begin{proof}
Let $\{\varepsilon_1, \varepsilon_2, \ldots, \varepsilon_{n-1}\}$ be the basis of $A_{n-1}^{*}$ defined in Subsection \ref {Def of Const A and B}.
For $x=x_1\varepsilon_1+\cdots+x_{n-1}\varepsilon_{n-1} \in A_{n-1}^*$, we have
\begin{equation}\label{innerproduct}
n(x,x)= \sum_{1\le i \le n-1}x_i^{2}+\sum_{1 \le i < j \le n-1}(x_i-x_j)^2.
\end{equation}
Let $S_1$ and $S_2$ denote the first and second summations in (\ref{innerproduct}), respectively.  

We suppose that $(x,x)>N_1(A_{n-1}^*)=(n-1)/n$ and $n \ge 7$. If $S_2$ is equal to zero, then we have $x_i=x_j$ for any $i$ and $j$. By assumption, since $|x_i| \ge 2$ for any $i$, we obtain $n(x,x) \ge 4(n-1)$, that is, $(x,x) \ge (4n-4)/n  > (n+1)/n$. Hence, we may assume that $S_2$ is not equal to zero.  By Lemma~\ref{lem}, we have
$ S_1 \ge 1$ and $S_2 \ge n-2$. Hence, we consider the cases where the number of $i$ such that $x_i^2=1$ is $k$ and the other coordinates are all zero, where $k \in \{2,3\}$.

\noindent (i) The case $k=2$

 We have $(\varepsilon_i-\varepsilon_j, \varepsilon_i-\varepsilon_j)=2$ and $(\varepsilon_i+\varepsilon_j, \varepsilon_i+\varepsilon_j)=(2n-4)/n$ for $i \ne j$. Moreover, we obtain $(2n-4)/n > (n+1)/n$ for  $n\ge 7$.

\noindent (ii) The case $k=3$

 If $i \ne j$, $i\ne l$, and $j \ne l$, then we have $(\varepsilon_i+\varepsilon_j+\varepsilon_l, \varepsilon_i+\varepsilon_j+\varepsilon_l) = (3n-9)/n$ and $(\varepsilon_i+\varepsilon_j-\varepsilon_l, \varepsilon_i+\varepsilon_j-\varepsilon_l)=(3n-1)/n.$ Moreover, we obtain  $(3n-9)/n,  (3n-1)/n > (n+1)/n$  for $n\ge 7$.
 
By the above discussion, we have the desired result.
\end{proof}
Let $p$ be an odd prime and $\alpha_1, \ldots, \alpha_{p-1}$, $\varepsilon_1$, $\rho_{\Delta}$ as in Subsection \ref{Def of Const A and B}.
\begin{lemma}\label{lower bound 2}
Let $l$ be an integer such that $1\le l \le p-1$. Then we have $N_1(\frac{l}{p}\rho_{\Delta}+A_{p-1}^*) \ge \frac{(p-1)(p+1)}{12p}$.
\end{lemma} 
\begin{proof}
Let $k$ be an integer such that $0 \le k \le p-1$. It suffices to prove that $N_1(k\varepsilon_1+\frac{l}{p}\rho_{\Delta}+A_{p-1})\ge  \frac{(p-1)(p+1)}{12p}$. Let $k_1, \ldots, k_{p-1} \in \mathbb{Z}$ and set $x=k_1\alpha_1+\cdots+k_{p-1}\alpha_{p-1}+k\varepsilon_1+\frac{l}{p}\rho_{\Delta}$.  Then we have 

\[(x,x)=(k_1+\frac{1}{p}(k(p-1)+\frac{p-1}{2}l))^2+\sum_{1\le i \le p-1}(k_{i+1}-k_{i}+\frac{1}{p}(-k+\frac{p-2i-1}{2}l))^2,\]
where $k_p=0$.
Note that 
\[\{(-k+\frac{p-2i-1}{2}l)\bmod p \mid 0 \le i \le p-1\}=\{\frac{p-2i-1}{2} \bmod p \mid  0 \le i \le p-1 \}.\]
Hence, we have 
\[(x,x) \ge \frac{2}{p^2}\sum_{1 \le i \le \frac{p-1}{2}}i^2= \frac{(p-1)(p+1)}{12p}.\]
\end{proof}

\section{Classification of isomorphism classes of lattices from Construction A }\label{Iso prob for Const A}
Let $p$ be a prime,  $C$ and $D$ self-orthogonal codes of length $k$ over $\mathbb{F}_p$. 
The goal in this section is to prove the following theorem:
\begin{theorem}\label{answer of isom prob for const A}
Let $p$ be an odd prime.
Then $L_A(C) \cong L_A(D)$ as lattices if and only if $C \cong D$ as codes.
\end{theorem} 
We start with preparations to prove the above theorem.

\subsection{$A_n$-frames of lattices}

In this subsection, we define  the notion of an $A_n$-frame of a lattice and describe some properties of $A_{p-1}$-frames of $L_A(C)$.
\begin{definition}
Let $m$ and $n$ be positive integers, and $L$ an even lattice. A finite subset $F$ of $L(2)$ is called  an \emph{$A_n$-frame} of $L$ if the following conditions hold:
\begin{enumerate}
\item $F$ forms a  root system of type $A_n^m$. 
\item $\operatorname{rank}L=mn$.
\end{enumerate}
\end{definition}
Let $R$ be the lattice $A_{p-1}^k$. By the definition of $L_A(C)$, we have $R(2) \subset L_A(C)$. Hence, $R(2)$ is an $A_{p-1}$-frame of $L_A(C)$.

By Lemmas~\ref{minimal norm} and \ref{lower bound}, we have the following lemma:
\begin{lemma}\label{non-trivial root}
$L_A(C)(2)=R(2)$ if $p \ge 7$.
\end{lemma}
\begin{proof}
Let $p$ be an odd prime such that $p \ge 7$. We determine all roots of $L_A(C)$. By Lemmas~\ref{minimal norm}~and~\ref{lower bound}, we have $N_1(A_{p-1}^{*})=(p-1)/p$ and $N_2(A_{p-1}^*)> (p+1)/p$. 
Moreover, there exist no positive integers $m$ such that $m(p-1)/p =2$. Since $A_{p-1}^*(2)=A_{p-1}(2)$, we have $L_A(C)=R(2)$.  
\end{proof}

Let $A_{p-1,i}$ denotes the lattice $A_{p-1}$ in the $i$-th coordinate of the lattice $R$. We define $\alpha_j^{i}$ and $\varepsilon_j^{i}$ by the elements $\alpha_j$ and $\varepsilon_j$ of $A_{p-1,i}$, respectively. 

Note that the automorphism group $\operatorname{Aut}(L_A(C))$ acts naturally on the set of $A_{p-1}$-frames of $L_A(C)$.
The following proposition implies that the structures of codes are recovered by the transitivity. 
\begin{proposition}\label{transitivity of frame implies isomorphic}
Suppose that  $\operatorname{Aut}(L_A(D))$ acts transitively on the set of $A_{p-1}$-frames of $L_A(D)$. If $L_A(C) \cong L_A(D)$ as lattices,  then we have $C \cong D$ as codes. 
\end{proposition}
\begin{proof}

Let $f : L_A(C) \rightarrow L_A(D)$ be an isomorphism of lattices. Then $f(R(2))$ is an $A_{p-1}$-frame of $L_A(D)$, where $R=A_{p-1}^k$.  By the assumption, there exists $\sigma \in \operatorname{Aut}(L_A(D))$ such that $\sigma f(R(2)) = R(2)$. Since $A_{p-1}(2)$ is an irreducible root system, $\sigma f$ induces a permutation of $\{A_{p-1,i}(2) \mid 1\le i \le k\}$.

Let $\sigma f(A_{p-1,i}(2)) = A_{p-1,\rho(i)}(2)$ and $\Delta_{i} = \{\alpha_j^{i} \mid 1\le j \le p-1\}$, where $\rho$ is a permutation of $\{1,2, \ldots, k\}$.  Then $\sigma f(\Delta_i)$ is a base of $A_{p-1, \rho(i)}(2)$. Since the Weyl group $W(A_{p-1, \rho(i)}(2))$ acts transitively on the set of bases of $A_{p-1, \rho(i)}(2)$, for each $i$, there exists $r_i \in W(A_{p-1, \rho(i)}(2))$ such that $r_i \sigma f(\Delta_i)=\Delta_{\rho(i)}$. This implies for each $i$, we have $r_i\sigma f(\alpha_j^{i})=\alpha_j^{\rho(i)}$ for any $j$ or $r_i\sigma f(\alpha_j^{i})=\alpha_{p-j}^{\rho(i)}$ for any $j$. Note that if the latter holds, then $r_i\sigma f(\varepsilon_1^{i}) \equiv -\varepsilon_1^{\rho(i)} \mod R$.  Hence, the map $r_1r_2 \cdots r_k \sigma f$ induces $C \cong D$ as codes.   
\end{proof}

By Proposition~\ref{transitivity of frame implies isomorphic}, it suffices to prove the transitivity of the automorphism group. By Lemma~\ref{non-trivial root}, we consider the cases $p=3,5$.

\subsection{The case $p=3$}
Following \cite{Kondo}, we explain this case.
Let $\alpha_i\ (1\le i \le 2)$ and $\varepsilon_j\ (1 \le j \le 3)$ be the vectors defined in Subsection \ref {Def of Const A and B}.
By (\ref{innerproduct}) and direct calculations, we have the following lemma:

\begin{lemma}\label{square norm of short vectors of A_2^{*}} 
If $x \in A_2^{*}$ and $0<(x,x) <2$, then we have $(x,x)=2/3$ and $A_2^{*}(2/3)= \{\pm\varepsilon_1,\pm\varepsilon_2, \pm\varepsilon_3\}$.
\end{lemma}
Let $X_1 =\{\varepsilon_1, \varepsilon_2, \varepsilon_3\}$ and $X_2 =\{-\varepsilon_1, -\varepsilon_2, -\varepsilon_3\}$.   Note that $A_2^{*}(2/3)=X_1 \cup X_2$ (disjoint) and  
\begin{equation}\label{mod A_2}
x \equiv i\varepsilon_1 \mod A_2
\end{equation}
for  $i \in \{1,2\}$ and $ x \in X_i $. 


\begin{lemma}\label{action of  the automorphism group of A_2}
The following hold:
\begin{enumerate}
\item Let $\sigma \in \operatorname{Aut}(A_2)$ be the map defined by $\alpha_1 \mapsto \alpha_2$ and $\alpha_2 \mapsto \alpha_1$.Then $\sigma(X_1) = X_2$.  
\item The Weyl group $W(A_2)$ acts transitively on $X_1$ and $X_2$. 
\item \label{the statemant for the automorphism group of A_2}  The group $\operatorname{Aut}(A_2)$ acts transitively on $A_2^{*}(2/3)$.
\end{enumerate}
\end{lemma}

\begin{proof}
(1) Since $\varepsilon_1 =(1/3)(2\alpha_1+\alpha_2)$, we have $\sigma(\varepsilon_1)=2\varepsilon_1-\alpha_1$.
By (\ref{mod A_2}), we have $\sigma(X_1)=X_2$.
 
\noindent (2) Let $r_{\alpha_i}\ (1 \le i \le 2)$ be the reflection induced by the simple root $\alpha_i$.  By the definition, $r_{\alpha_1}$ permutes $\varepsilon_1 , \varepsilon_2$ and  fixes $\varepsilon_3$. Furthermore, $r_{\alpha_2}$ permutes $\varepsilon_2,\ \varepsilon_3$ and fixes $\varepsilon_1$.  Hence, we can identify the generated group $\langle r_{\alpha_1}, r_{\alpha_2} \rangle \subset W(A_2)$ with the symmetric group of the set $X_1$. This means that $W(A_2)$ acts transitively on $X_1$. Moreover, since $\sigma(X_1)=X_2$ and $\sigma W(A_2) \sigma^{-1}=W(A_2)$, the transitivity of $W(A_2)$ on $X_1$ implies the transitivity of $W(A_2)$ on $X_2$. 

\noindent (3) The transitivity of $\operatorname{Aut}(A_2)$ follows from (1) and (2). 
\end{proof}

\begin{proposition}\label{transitivity for p=3}
The group $\operatorname{Aut}(L_A(C))$ acts transitively on the set of $A_2$-frames of $L_A(C)$.
\end{proposition}

\begin{proof}
	Let $F$ be an $A_2$-frame of $L_A(C)$ and $R=A_2^k$. Then $F$ is expressed as $F=F_1\cup \cdots \cup F_k$ (disjoint), where $F_i \perp F_j$ for $i \ne j$ and $F_i$ forms a root system of type $A_2$ for each $i$. To prove the transitivity, we prove that if $\#(F \cap R(2)) < \#R(2)$, then there exists $\psi \in \operatorname{Aut}(L_A(C))$ such that $ \#(F \cap R(2))<\#(\psi(F) \cap R(2))$. 

Let $\lambda \in F \setminus R(2)$ and we may assume $\lambda \in F_1$. By Lemmas \ref{square norm of short vectors of A_2^{*}} and \ref{action of  the automorphism group of A_2} (\ref{the statemant for the automorphism group of A_2}), there exist $c \in C(3)$ and $\varphi \in \operatorname{Aut}(R)$ such that $\varphi(\lambda) = \sum_{i \in \operatorname{Supp}c}\varepsilon_1^{i}$. Then there exists $i_1 \in \operatorname{Supp}c$ such that $\varphi(F_1) \cap R(2) \subset \{\pm \alpha _1^{i_1}\}$ or $\varphi(F_1) \cap R(2) \subset \{\pm (\alpha _1^{i_1}+\alpha_2^{i_1})\}$. By considering the reflection induced by $\alpha_2^{i_1}$, we may assume $\varphi(F_1) \cap R(2) \subset \{\pm \alpha _1^{i_1}\}$.

\noindent (i) The case  $\varphi(F_1) \cap R(2) = \{\pm \alpha _1^{i_1}\}$	

In this case, we have $(\varphi(F) \setminus \varphi(F_1)) \cap (\bigcup_{i \in \operatorname{Supp c}}A_{2,i}(2)) \subset \{\pm \alpha_2^{i} \mid  i \in \operatorname{Supp}c \setminus \{i_1\} \}$. Let $v= \varphi(\lambda) - \alpha_1^{i_1}-\alpha_2^{i_1}$ and $r_v$ the reflection induced by the vector $v$. By the definition of reflection, we see that $r_v(x) = x- (x,v)v$ and $r_v \in \operatorname{Aut}(\varphi(L_A(C)))$. By direct calculations, we have $r_v(x)=x$ for any $x \in \varphi(F) \cap R(2)$ and $r_v(\varphi(\lambda)) = \alpha_1^{i_1}+\alpha_2^{i_1}$. Hence, we have $\varphi^{-1} r_v \varphi \in \operatorname{Aut}(L_A(C))$ such that $\#(\varphi^{-1} r_v \varphi(F) \cap R(2)) > \#(F \cap R(2))$.

\noindent (ii) The case  $\varphi(F_1) \cap R(2) = \emptyset$	

We first prove that there exist $j_1, j_2 \in \operatorname{Supp}c$ such that  
\begin{equation}\label{condition for roots}
\varphi(F) \cap (\bigcup\nolimits_{i \in \operatorname{Supp}c}A_{2,i}(2)) \subset \{\pm \alpha_2^{j_1}, \pm \alpha_2^{j_2}\}.
\end{equation}
To prove (\ref{condition for roots}), we suppose that
$\varphi(F) \cap (\bigcup_{i \in \operatorname{Supp}c}A_{2,i}(2)) =   \{\pm \alpha_2^{i} \mid  i \in \operatorname{Supp}c\}$. Since  $\varphi(F_1) \cap R(2) = \emptyset$, there exists $\lambda' \in \varphi(F_1)$ such that $(\varphi(\lambda),\lambda') = -1$ and $\lambda' \notin R(2)$. Since $\lambda' \perp  \{\pm \alpha_2^{i} \mid  i \in \operatorname{Supp}c\}$, we see that for each $i \in \operatorname{Supp}c$, the $i$-th coordinate of $\lambda'$ is equal to $0$ or $\varepsilon_1$ or $-\varepsilon_1$. By (\ref{varepsilon:inner product}), we have $(\varepsilon_1, \varepsilon_1)=2/3$ and $(\varepsilon_1, -\varepsilon_1)=-2/3$. This implies  $(\varphi(\lambda),\lambda') \neq -1$, which contradicts  $(\varphi(\lambda),\lambda') = -1$. Hence (\ref{condition for roots}) holds. 
Let $r_v$ be the reflection induced by the vector $v=\varphi(\lambda)  - \alpha_1^{j_3}-\alpha_2^{j_3}$, where $j_3 \in \operatorname{Supp}c \setminus \{j_1, j_2 \}$. Similarly, we see that  $\varphi^{-1} r_v \varphi \in \operatorname{Aut}(L_A(C))$ and $\#(\varphi^{-1} r_v \varphi(F) \cap R(2)) > \#(F \cap R(2))$.
\end{proof}

\subsection{The case $p=5$} 
Let $\alpha_i\ (1\le i \le 4)$ and $\varepsilon_j\ (1 \le j \le 5)$ be the vectors defined in Subsection \ref {Def of Const A and B}.
By (\ref{innerproduct}) and direct calculations, we have the following lemma:
\begin{lemma}\label{square norm of short vectors of A_4^{*}} 
If $x \in A_4^{*}$ and $0<(x,x) \le 6/5$, then $(x,x) = 4/5$ or $6/5$. Moreover, we have
\[
\begin{aligned}[t]
A_4^{*}(4/5) &= \{\pm \varepsilon_i \mid 1 \le i \le 5\},\\
A_4^{*}(6/5) &= \{\pm(\varepsilon_i+\varepsilon_j) \mid 1 \le i < j \le 4\}\\
             &\quad \cup \{\pm (\varepsilon_i+\varepsilon_j+\varepsilon_k)
             \mid 1 \le i < j < k \le 4\}.
\end{aligned}
\]
\end{lemma}

Let $X_1$ and $X_2$ be the following sets:
\[
\begin{aligned}[t]
X_1 &= \{\varepsilon_i \mid 1 \le i \le 5\},\\
X_2 &= \{\varepsilon_i+\varepsilon_j \mid 1 \le i < j \le 4\}\\
    &\quad \cup \{-(\varepsilon_i+\varepsilon_j+\varepsilon_k)
    \mid 1 \le i < j < k \le 4\}.
\end{aligned}
\]
Note that $A_4^{*}(4/5) = X_1 \cup (-X_1)$ (disjoint) and $A_4^{*}(6/5) = X_2 \cup (-X_2)$ (disjoint). Moreover, we have  
\begin{equation}\label{mod A_4}
x \equiv i\varepsilon_1 \mod A_4\ \text{and}\ 
-x \equiv (5-i)\varepsilon_1 \mod A_4
\end{equation} 
for $i \in \{1,2\}$ and $ x \in X_i $. 


\begin{lemma}\label{transitivity for A_4}
The following hold:
\begin{enumerate}
\item \label{(1)} Let $\sigma \in \operatorname{Aut}(A_4)$ be the map defined by $\alpha_1 \mapsto \alpha_4$, $\alpha_2 \mapsto \alpha_3$, $\alpha_3 \mapsto \alpha_2$, $\alpha_4 \mapsto \alpha_1$. Then $\sigma(X_1)=-X_1$ and $\sigma(X_2)=-X_2$.
\item \label{(2)} The Weyl group $W(A_4)$ acts transitively on $X_1$ and $X_2$.
\item \label{the statemant for the automorphism group of A_4} The group $\operatorname{Aut}(A_4)$ acts transitively on  $A_4^{*}(4/5)$ and $A_4^{*}(6/5)$.
\end{enumerate}
\end{lemma}

\begin{proof}
(1) Since $\varepsilon_1=(1/5)(4\alpha_1+3\alpha_2+2\alpha_3+\alpha_4)$, we have $\sigma(\varepsilon_1)=4\varepsilon_1-3\alpha_1-2\alpha_2-\alpha_3$. By (\ref{mod A_4}), we obtain $\sigma(X_1)=-X_1$ and $\sigma(X_2)=-X_2$.

\noindent (2) Let $r_{\alpha_i}\ (1 \le i \le 4)$ be the reflection induced by the simple root $\alpha_i$. We first consider the transitivity of $W(A_4)$ on $X_1$. By direct calculations, we see that $r_{\alpha_i}$ permutes $\varepsilon_i,\ \varepsilon_{i+1}$ and fixes the other elements of $X_1$. Hence, we can identify the generated group $\langle r_{\alpha_i} \mid 1 \le i \le 4 \rangle$ with the symmetric group of $X_1$. This implies the transitivity of $W(A_4)$ on $X_1$. 

Next, we consider the transitivity of $W(A_4)$ on $X_2$. Let 
\[
\begin{aligned}
x_1'&=-(\varepsilon_1+\varepsilon_2+\varepsilon_3),\ x_2'=-(\varepsilon_1+\varepsilon_2+\varepsilon_4),\\ 
x_3'&=-(\varepsilon_1+\varepsilon_3+\varepsilon_4),\ x_4'=-(\varepsilon_2+\varepsilon_3+\varepsilon_4)
\end{aligned}
\]
and 
\[
\begin{aligned}
x_1''&=\varepsilon_1+\varepsilon_2,\ x_2''=\varepsilon_1+\varepsilon_3,\\ 
x_3''&=\varepsilon_1+\varepsilon_4,\ x_4''=\varepsilon_2+\varepsilon_3,\\ 
x_5''&=\varepsilon_2+\varepsilon_4,\ x_6''=\varepsilon_3+\varepsilon_4.
\end{aligned}
\]
We define $X_2'$ and $X_2''$ by  
\[X_2'= \{x_i' \mid 1 \le i \le 4\}\ \text{and}\ X_2''=\{x_i'' \mid 1 \le i  \le 6\}.\] 
Then $X_2 = X_2' \cup X_2''.$ 
By direct calculations, the reflections $r_{\alpha_i}\ (1 \le i \le 4)$ are identified with permutations of $X_2$ as follows:
\[r_{\alpha_1} : (x_3'\ x_4')(x_2''\ x_4'')(x_3''\ x_5''),\ r_{\alpha_2}: (x_2'\ x_3')(x_1''\ x_2'')(x_5''\ x_6''),\]
\[ r_{\alpha_3}: (x_1'\ x_2')(x_2''\ x_3'')(x_4''\ x_5''),\ r_{\alpha_4}: (x_2'\ x_6'')(x_3'\ x_5'')(x_4'\ x_3'').\]  Hence, we see that the generated group $\langle r_{\alpha_i} \mid 1 \le i \le 3 \rangle \subset W(A_4)$ acts transitively on $X_2'$ and $X_2''$,  and $r_{\alpha_4}$ permutes elements of $X_2'$ and $X_2''$. This implies that $W(A_4)$ acts transitively on $X_2$.  

\noindent (3) Since $A_4^{*}(4/5) = X_1 \cup (-X_1)$ and $A_4^{*}(6/5) = X_2 \cup (-X_2)$, by (\ref{(1)}) and (\ref{(2)}), we have the desired result.   
\end{proof}

Let $\lambda$ be an element of $(A_4^{*})^{2}(2) \setminus A_4^{2}(2)$. We define a lattice $L_{\lambda}$ by $L_{\lambda} = \mathbb{Z}\lambda + A_4^2$.  Then the following lemmas hold:
\begin{lemma}\label{E_8}
The lattice $L_{\lambda}$ is isomorphic to the root lattice $E_8$.
\end{lemma}
\begin{proof}
Since  $(\lambda, \lambda) =2$,  we see that $L_{\lambda}$ is even.  Moreover, since $A_4^{2} \subsetneq L_{\lambda} \subset L_{\lambda}^{*} \subsetneq  (A_4^{*})^{2}$ and  $|(A_4^{*})^{2}/A_4^{2}| = 5^2$, we have $|L_{\lambda}^{*}/L_{\lambda}| =1$. Hence, the lattice $L_{\lambda}$ is a positive-definite even unimodular lattice of rank $8$, which implies $L_{\lambda} \cong E_8$ as lattices. 
\end{proof}

\begin{lemma}\label{transitivity for E_8}
Let $X = \{\{x_1, x_2, x_3, x_4\} \subset E_8(2) \mid  (x_1, x_2)=(x_2, x_3)=(x_3, x_4)=-1,  (x_i, x_j) =0\ \text{otherwise} \}$. Then the Weyl group $W(E_8)$ of the root system acts transitively on $X$. 
\end{lemma}
\begin{proof}
Let $E_8=\{(x_1, \ldots, x_8) \in \mathbb{Z}^8 \cup (\mathbb{Z} +1/2)^8 \mid x_1+\cdots + x_8 \in 2\mathbb{Z}\}.$
First, we determine the number of the elements of $X$. By \cite[Section 10.4, Lemma C]{Hu}, since the group $W(E_8)$ acts transitively on $E_8(2)$, we fix $x_1=(1,1,0,0,0,0,0,0).$ By direct counting, we have $\#\{\{x_2, x_3, x_4\} \subset E_8(2) \mid  (x_1, x_2)=(x_2, x_3)=(x_3, x_4)=-1,  (x_i, x_j) =0\ \text{otherwise} \}=24192.$ Hence, we have $\#X=1/2 \cdot \#E_8(2) \cdot 24192=2903040.$ 

Next, we prove that there exists $S \in X$ such that $|\operatorname{Stab}_{W(E_8)}(S)| \le 240.$ 
Let
\[
\begin{aligned}
&a_1=(1,-1,0,0,0,0,0,0),\ a_2=(0,1,-1,0,0,0,0,0),\\
&a_3=(0,0,1,-1,0,0,0,0),\ a_4=(0,0,0,1,-1,0,0,0),\\
&a_5=(0,0,0,0,1,-1,0,0),\ a_6=(0,0,0,0,0,1,1,0),\\
&2a_7=(-1,-1,-1,-1,-1,-1,-1,-1),\ a_8=(0,0,0,0,0,1,-1,0),\\
&\theta = 2a_1+3a_2+4a_3+5a_4+6a_5+4a_6+3a_8+2a_7.
\end{aligned}
\]
Then $\Delta_{E_8(2)}=\{a_1,\ldots,a_8\}$ is a base of the root system $E_8(2)$ and $\theta$ is the highest root.
Let $S = \{-\theta, a_1,a_2,a_3\} \in X$ and $S^{\perp} = \{x \in E_8(2) \mid (x,y)=0\ \text{for any $y \in S$}\}.$ 
Since $\#S^{\perp}=20$ and $a_5,a_6,a_7,a_8 \in S^{\perp}$, $S^{\perp}$ is a root system of type $A_4$ with a base $\Delta_{S^{\perp}}=\{a_8, a_5, a_6, a_7\}.$ Then we have a group homomorphism $f: \operatorname{Stab}_{W(E_8)}(S) \rightarrow \operatorname{Aut}(S^{\perp})\ (\sigma \mapsto \sigma |_{S^{\perp}})$, where $\operatorname{Aut}(S^{\perp})$
is the automorphism group of the root system $S^{\perp}$. Since $\operatorname{Im}f$ includes the Weyl group $W(A_4) \cong S_5$ of the root system $S^{\perp}$ and $|\operatorname{Aut}(S^{\perp})|=240$, we have $|\operatorname{Im}f| = 120,240$. Moreover, we have $|\operatorname{Ker}f| \le 2$. By directly verifying the elements of $E_8(2)$, we see that the case $|\operatorname{Ker}f| =  2$ and $|\operatorname{Im}f| =240$ does not occur. Hence, we have
$|\operatorname{Stab}_{W(E_8)}(S)| \le 240.$ Since $\#\{g(S) \mid g \in W(E_8)\} \ge |W(E_8)|/240=2903040=\#X$, we obtain the desired transitivity. 
\end{proof}

\begin{proposition}\label{transitivity for p=5}
The group $\operatorname{Aut}(L_A(C))$ acts transitively on the set of $A_4$-frames of $L_A(C)$.
\begin{proof}
Let $F$ be an $A_4$-frame of $L_A(C)$ and $R=A_4^{k}$. Then $F$ is expressed as $F=F_1\cup \cdots \cup F_k$ (disjoint), where $F_i \perp F_j$ for $i \ne j$ and $F_i$ forms a root system of type $A_4$. To prove the transitivity, we prove that if $\#(F \cap R(2)) < \#R(2)$, then there exists $\psi \in \operatorname{Aut}(L_A(C))$ such that $ \#(F \cap R(2))<\#(\psi(F) \cap R(2))$. 

Let $\lambda \in F \setminus R(2)$ and we may assume $\lambda \in F_1$. By Lemmas \ref{square norm of short vectors of A_4^{*}} and \ref{transitivity for A_4} (\ref{the statemant for the automorphism group of A_4}), there exist $c \in C(2)$ and $\varphi \in \operatorname{Aut}(R)$ such that $\varphi(\lambda) = \varepsilon_1^{i_1}+2\varepsilon_1^{i_2}-\alpha_1^{i_2}$, where  $\operatorname{Supp} c = \{i_1,i_2\}$. 

Let $L_{\varphi(\lambda)} = \mathbb{Z}\varphi(\lambda) + A_{4,i_1}+A_{4,i_2}$,  where $A_{4,i}$ denotes the $i$-th lattice $A_{4}$ of $R=A_4^k$. Since $\varphi(F_1)$ forms a irreducible root system  of type $A_4$, we have
\begin{equation*}\label{condition for A_4} 
  \varphi(F_1) \subset L_{\varphi(\lambda)}. 
\end{equation*}  
 
By Lemma \ref{E_8}, we have $L_{\varphi(\lambda)} \cong E_8$ as lattices. Hence, by Lemma \ref{transitivity for E_8}, there exists $\sigma \in  W(L_{\varphi(\lambda)}) \subset \operatorname{Aut}(\varphi(L_A(C)))$ such that $\sigma\varphi(F_1)=A_{4,i_1}(2)$ and $\sigma(x) =x $ for $x \in R(2) \setminus \bigcup_{i \in \operatorname{Supp}c} A_{4,i}(2)$. Moreover, since $(\sigma \varphi(F) \setminus \sigma \varphi(F_1)) \perp \sigma \varphi(F_1)$, if $x \in (\varphi(F) \setminus \varphi(F_1)) \cap  (\bigcup_{i \in \operatorname{Supp}c} A_{4,i}(2))$, then $\sigma (x) \in  A_{4, i_2}(2) \subset R(2)$. 
Hence, we have 
\[\#((\sigma \varphi(F) \setminus \sigma \varphi(F_1)) \cap R(2)) \ge \#((\varphi(F) \setminus  \varphi(F_1)) \cap R(2)).\] Combining this with $\#(\sigma \varphi(F_1) \cap R(2)) > \#(\varphi(F_1) \cap R(2))$, we see $\#(\varphi^{-1} \sigma \varphi(F) \cap R(2)) > \#(F \cap R(2))$ for $\varphi^{-1} \sigma \varphi \in \operatorname{Aut}(L_A(C))$.
\end{proof}
\end{proposition}

\subsection{Conclusion of Section \ref{Iso prob for Const A}}
By Lemma~\ref{non-trivial root}, if $p \ge 7$, then $A_{p-1}^{k}(2)$ is the only $A_{p-1}$-frame of $L_A(C)$. Hence,  $\operatorname{Aut}(L_A(C))$ also acts transitively on the set of $A_{p-1}$-frames of $L_A(C)$ for $p \ge 7$. Combining this with Propositions  \ref{transitivity for p=3} and \ref{transitivity for p=5}, we have the following proposition:

\begin{proposition}\label{transitivity for p}
Let $p$ be an odd prime and  $C,D$ self-orthogonal codes of length $k$ over $\mathbb{F}_p$.
Then the group $\operatorname{Aut}(L_A(C))$ acts transitively on the set of $A_{p-1}$-frames of $L_A(C)$.
\end{proposition}

Hence, by Proposition \ref{transitivity of frame implies isomorphic}, we have Theorem \ref{answer of isom prob for const A}.

\section{Classification of isomorphism classes of lattices from Construction B}
Let $p$ be a prime and $C, D$ self-orthogonal codes of length $k$ over $\mathbb{F}_p$. 
The goal in this section is to prove the following theorem:
\begin{theorem}\label{Answer of iso prob for Const B}
Let $p$ be an odd prime. Then $L_B(C) \cong L_B(D)$ as lattices if and only if $C \cong D$ as codes.
\end{theorem}

We first start with preparations to prove the above theorem.

\subsection{Preparations for the case $p=3$}
Let $d_3 = \langle (1, 1, 1) \rangle \subset \mathbb{F}_3^3$ and $K_3 = \langle (1, 2, 0) \rangle \subset \mathbb{F}_3^{3}$.  Define a map $\bar{\cdot}: d_3 \rightarrow \mathbb{F}_3$ by $a (1, 1, 1) \mapsto a$. Let $m$ be a positive integer and set $(d_3^{m})_{0} = \{(x_1, \ldots, x_m) \in d_3^{m} \mid \bar{x}_1+\cdots+\bar{x}_m = 0\}$. We identify the vector spaces  $d_3^{m}$ and $K_3^{m}$ with subsets of $\mathbb{F}_3^{3m}$.

Let $K$ be a code of $K_3^{m}$.   Then we define $C_A(K)$ and $C_B(K)$ called the codes constructed by \emph{Construction A} from $K$ and \emph{Construction B} from $K$, respectively, as follows:
\[C_A(K) = K+ d_3^{m},\ C_B(K) = K + (d_3^{m})_0.\]

The above constructions are analogues of Construction A and B from a Kleinian code in \cite[Subsection 1.5]{Shima}. By the definition of Construction A, we have the following proposition:
\begin{proposition}\label{condition that Const A is even}
Let $K \subset K_3^m$ be a code. Then $K$ is self-orthogonal if and only if $C_A(K)$ is self-orthogonal.   
\end{proposition}

For $1 \le i \le m$, we define
\begin{equation}\label{def of u}  
u_i = (u_{i,1}, u_{i,2}, \ldots, u_{i,3m}) \in \mathbb{F}_3^{3m} 
\end{equation}
by $u_{i,3i-2}=u_{i,3i-1}=u_{i,3i}=1$ and $u_{i,j}=0$ otherwise. 
Note that $d_3^m=\langle u_i \mid 1 \le i \le m \rangle$ and $(d_3^m)_0=\langle u_i+2u_m \mid 1 \le i \le m\rangle$.
In order to give conditions for a code over $\mathbb{F}_3$ to be realized by Construction B, we prove the following lemma:
\begin{lemma}\label{coordinate permutation in p=3}
Let $C$ be a self-orthogonal code of length $3m$ over $\mathbb{F}_3$. Suppose that $(d_3^m)_0 \subset C$ and  $u_1 \in C^{\perp} \setminus C$. Then $u_i \in C^{\perp}$  for each $1 \le i \le m$ and there exist $y' \in C^{\perp}$ and a coordinate permutation $\sigma$ such that the following conditions hold:
\begin{enumerate}
\item $\sigma(u_i)=u_i$ for each $1 \le i \le m$.
\item $\sigma(y')=(0,0,1,0,0,1,\ldots,0,0,1).$ 
\end{enumerate}  
\end{lemma}
\begin{proof}
 Let $\widetilde{C} = C+\mathbb{F}_3u_1$. Since $C$ is self-orthogonal and $u_1 \in C^{\perp}$, the code $\widetilde{C}$ is also self-orthogonal. Furthermore, since $u_1 \notin C$, we have $C \subsetneq \widetilde{C}$. 
 
Let $ y=(y_1,y_2,\ldots,y_m) \in C^{\perp} \setminus \widetilde{C}^{\perp}$, where $y_i \in \mathbb{F}_3^3$ for each $i$. Since $(d_3^m)_0 \subset C$, we have $(y, u_i+u_j+u_k)=0$ for any $i, j, k$. For a contradiction, suppose that
$(y, u_i+u_j) = 0$ for some $i,j$. Then we have $(y, u_1)=2(y,u_i+u_j+2u_i+2u_j+2u_1)=0$, which contradicts  $y \notin \widetilde{C}^{\perp}$. Hence, we see $(y, u_i+u_j) \neq 0$ for any $i,j$. 

From the condition $(y, u_i+u_j) \neq 0$ for any $i,j$, it follows
that  
\[(y_i, (1,1,1)) =1\ \text{for any $i$ or}\ (y_i, (1,1,1)) =2\ \text{for any $i$.}\]
 Since $-y \in C^{\perp} \setminus \widetilde{C}^{\perp}$, we may assume $(y_i, u_i)=1$ for any $i$. 
By this condition, for each $i$, there exists $a_i \in \mathbb{F}_3$ such that $y_i+a_i(1,1,1) \in \{(1,0,0),(0,1,0),(0,0,1)\}$. Note that  $u_i \in C^{\perp}$ for each $i$, since $2u_1+u_i \in (d_3^m)_0 \subset C \subset C^{\perp}$ and $u_1 \in C^{\perp}$. Considering the element $y'=y+a_1u_1+\cdots+a_mu_m \in C^{\perp}$, we have the desired result.
\end{proof}

The following proposition gives conditions for a code over $\mathbb{F}_3$ to be realized by Construction A and B:

\begin{proposition}\label{conditions for code: Const A and B}
Let $C$ be a self-orthogonal code of length $3m$ over $\mathbb{F}_3$. Then the following hold:
\begin{enumerate}
\item If $d_3^{m} \subset C$, then there exists a self-orthogonal code $K$ of $K_3^m$ such that $C \cong C_A(K)$ as codes.
\item \label{conditions for code: Const B} Suppose that $(d_3^m)_0 \subset C$ and $u_1 \in C^{\perp} \setminus C$. Then there exists a self-orthogonal code $K$ of $K_3^m$ such that $C \cong C_B(K)$ and $C+\mathbb{F}_3u_1 \cong C_A(K)$ as codes. 
\end{enumerate}
\end{proposition} 

\begin{proof}
(1) Let $x = (x_1, x_2, \ldots, x_m) \in C$, where $x_i \in \mathbb{F}_3^{3}$ for each $i$. Since $d_3^{m} \subset C \subset C^{\perp}$, we have $(x_i, (1,1,1)) =0$ for $1 \le i \le m$. Since $\{(1,2,0), (1,1,1)\}$ is a basis of $\{y \in \mathbb{F}_3^3 \mid (y, (1,1,1))=0\}$, for each $i$, there exists only one $x_i' \in \langle (1,2,0) \rangle$ such that $x_i \in x_i' + \langle (1,1,1) \rangle$. 

Define a map $\pi: C \rightarrow \mathbb{F}_3^{3m}$ by $x \mapsto (x_1',\ldots, x_m')$ and set $K=\{\pi(x) \mid x \in C\}$.  Then $K$ is a code of $K_3^{m}$. Moreover, since $d_3^{m} \subset C$, we have $C_A(K) = C$. In particular, $K$ is self-orthogonal and $C \cong C_A(K)$ holds as codes.

\noindent (2) Let $x \in C$. By Lemma \ref{coordinate permutation in p=3}, there exists $y' \in C^{\perp}$ and a coordinate permutation $\sigma$ such that the conditions in Lemma \ref{coordinate permutation in p=3}  holds. Since $u_i \in \sigma(C^{\perp}) \subset \sigma(C)^{\perp}$ for $1 \le i \le m$, we can define $\pi: \sigma(C) \rightarrow \mathbb{F}_3^{3m}$ as in the above map  $\pi$.   

Since $(\sigma(x),\sigma(y'))=0$ and $(\pi(\sigma(x)),\sigma(y'))=0$, we have $\sigma(x) \in \pi(\sigma(x))+(d_3^m)_0$. Since $(d_3^m)_0 \subset \sigma(C)$, by the same discussion in (1), we see that $\sigma(C) = C_B(K)$ and $\sigma(C+\mathbb{F}_3u_1)=C_A(K)$, where $K=\{\pi(\sigma(x)) \mid x \in C\}$. Hence, $C \cong C_B(K)$ and $C+\mathbb{F}_3u_1 \cong C_A(K)$ hold as codes.
\end{proof}

By using the above proposition,  we have another criterion for a code over $\mathbb{F}_3$ to be realized by Construction A and B. 
\begin{lemma}\label{another criterion for Const B in p=3}
Let $C$ be a self-orthogonal code of length $k$ over $\mathbb{F}_3$. Suppose that $C$ satisfies the following condition:
\begin{equation}\label{condition for another criterion p=3} 
u_1 \in C^{\perp}\setminus C\ \text{and}\  k/3+\#C(3) \le \#(u_1+C)(3). 
\end{equation}
Then we have $ 3 \mid k $ and there exists a self-orthogonal code $K$ of $K_3^{k/3}$ such that $C \cong C_B(K)$ and $C+\mathbb{F}_3u_1 \cong C_A(K)$ as codes.
\end{lemma}
\begin{proof}
In order to apply Proposition \ref{conditions for code: Const A and B} (\ref{conditions for code: Const B}), we show that
$3 \mid k$ and there exists a self-orthogonal code $C' \cong C$ which satisfies (\ref{condition for another criterion p=3}) and $\{u_i \mid 1 \le i \le k/3\} \subset u_1+C'$. To do this, we prove that if $\{u_i \mid 1 \le i \le r\} \subset u_1 +C$ for $r < k/3$, then there exists a self-orthogonal code $C''\cong C$ which satisfies  (\ref{condition for another criterion p=3}) and  $\{u_i \mid 1 \le i \le r+1\} \subset u_1 +C''$. 

Let $\{u_i \mid 1 \le i \le r\} \subset u_1 +C$ for  $r < k/3$ and 
 $x=(x_1,\ldots, x_r,y) \in (u_1+C)(3)$, where $x_i \in \mathbb{F}_3^3$ for each $i$ and $y \in \mathbb{F}_3^{k-3r}$. 
Since $u_i - x \in C$ and $u_i \in C^{\perp}$ for $1 \le i \le r$, we have $(u_i-x, u_i)=0$ for $1 \le i \le r$, which implies that $(x_i, (1,1,1))=0$ for $1 \le i \le r$. Hence, if $x_i \neq 0$ for some $1 \le i \le r$,  then 
$x = u_i$ or
$x_i \in \langle (1,1,1) \rangle^{\perp} (2)$ and $\#\operatorname{Supp}y=1$. 

Let $X_r$ be the set of elements $x=(x_1,\ldots, x_r,y) \in (u_1+C)(3)$ such that  there exists $i \in \{1,2,\ldots r\}$ such that 
\begin{equation}\label{conditions for another criterion for Const B in p=3}
x_i \in \langle (1,1,1) \rangle^{\perp} (2)\ \text{and}\ \#\operatorname{Supp}y=1.
\end{equation} 
Then we have $\#X_r \le \#C(3)$. Indeed, we can construct an injection $f: X_r \rightarrow C(3)$ as follows.  
For $x=(x_1,\ldots, x_r,y) \in X_r$, there exists only one  $i \in \{1,2,\ldots r\}$ such that (\ref{conditions for another criterion for Const B in p=3}) holds. We define an injection
 $f: X_r \rightarrow C(3)$ by $x \mapsto x-u_{i}$. 
 
Let $U_r=\{u_i \mid 1 \le i \le r\}$. Then we have the following: 
\begin{align*}
\#\{(0,\ldots,0, y) \in (u_1+C)(3) \mid y \in \mathbb{F}_3^{k-3r}\}
&=\#(u_1+C)(3)-\#U_r-\#X_r \\
&>\#(u_1+C)(3)-k/3-\#C(3) \\ 
&\ge 0.
\end{align*}
Hence, there exists a self-orthogonal code $C''\cong C$ which satisfies  (\ref{condition for another criterion p=3}) and  $\{u_i \mid 1 \le i \le r+1\} \subset u_1 +C''$.  This implies $3 \mid k$. Moreover, by Proposition \ref{conditions for code: Const A and B}, we have  $C \cong C_B(K)$ and $C+\mathbb{F}_3u_1 \cong C_A(K)$ as codes. 
\end{proof}

Let us establish an isomorphism between the lattices $L_A(C_B(K))$ and $L_B(C_A(K))$, where $K\subset K_3^m$ is a self-orthogonal code. 
Let $\alpha_i\ (1\le i \le 2)$ and $\varepsilon_j\ (1 \le j \le 3)$ be as in Subsection \ref{Def of Const A and B}. 
Note that $\varepsilon_1=(1/3)(2\alpha_1+\alpha_2)$ and $\varepsilon_i =\varepsilon_1-\alpha_1-\cdots-\alpha_{i-1}$ for $2 \le i \le 3$.
We define a linear map $\varphi:\bigoplus_{i=1}^3 \langle \alpha_1, \alpha_2 \rangle_{\mathbb{R}} \rightarrow \bigoplus_{i=1}^3  \langle \alpha_1, \alpha_2 \rangle_{\mathbb{R}}$ by 
\[
\begin{aligned}
&\varphi((\alpha_1,0,0))=(\varepsilon_3, \varepsilon_2,\varepsilon_1),
\varphi((\alpha_2,0,0))=(\varepsilon_2, \varepsilon_1, \varepsilon_3),\\
&\varphi((0,\alpha_1,0))=(\varepsilon_1, \varepsilon_2, \varepsilon_3),
\varphi((0,\alpha_2,0))=(\varepsilon_3, \varepsilon_1, \varepsilon_2),\\ 
&\varphi((0,0,\alpha_1))=(\varepsilon_1, \varepsilon_1, \varepsilon_1),
\varphi((0,0,\alpha_2))=(\varepsilon_3, \varepsilon_3, \varepsilon_3).
\end{aligned}
\]
By using (\ref{varepsilon:inner product}), we can directly check that $\varphi$ preserves the inner product $(\cdot,\cdot).$
Let $m$ be a positive integer.  We define a map $\widetilde{\varphi}: \bigoplus_{j=1}^m(\bigoplus_{i=1}^3 \langle \alpha_1, \alpha_2 \rangle_{\mathbb{R}}) \rightarrow \bigoplus_{j=1}^m(\bigoplus_{i=1}^3 \langle \alpha_1, \alpha_2 \rangle_{\mathbb{R}})$ by 
\[(x_1, \ldots, x_m) \mapsto (\varphi(x_1), \ldots, \varphi(x_m)).\] 
Let $\alpha_j^{i}$ and $\varepsilon_j^{i}$ denotes the element $\alpha_j$ and $\varepsilon_j$ in the $i$-th coordinate, respectively. 

\begin{lemma}\label{Const B p=3 constA} 
Let $K\subset K_3^m$ be a self-orthogonal code.  Set $C = C_A(K)$.
Then we have $\widetilde{\varphi}(\alpha_l^i) \in L_B(C)$  for $ 1 \le i \le 3m$ and $1 \le l \le 2$.
\end{lemma}

\begin{proof}
Let $i \in \{3k+1 \mid 0 \le k \le m-1\}$. By the definition of $\widetilde{\varphi}$ and direct calculations, we have 
\[
\begin{aligned}
&\widetilde{\varphi}(\alpha_1^i)=(\varepsilon_1^i+\varepsilon_1^{i+1}+\varepsilon_1^{i+2})+(-\alpha_1^i-\alpha_2^i-\alpha_1^{i+1}),\\
&\widetilde{\varphi}(\alpha_2^i)=(\varepsilon_1^i+\varepsilon_1^{i+1}+\varepsilon_1^{i+2})+(-\alpha_1^i-\alpha_1^{i+2}-\alpha_2^{i+2}),\\
&\widetilde{\varphi}(\alpha_1^{i+1})=(\varepsilon_1^i+\varepsilon_1^{i+1}+\varepsilon_1^{i+2})+(-\alpha_1^{i+1}-\alpha_1^{i+2}-\alpha_2^{i+2}),\\
&\widetilde{\varphi}(\alpha_2^{i+1})=(\varepsilon_1^i+\varepsilon_1^{i+1}+\varepsilon_1^{i+2})+(-\alpha_1^i-\alpha_2^i-\alpha_1^{i+2}),\\ 
&\widetilde{\varphi}(\alpha_1^{i+2})=(\varepsilon_1^i+\varepsilon_1^{i+1}+\varepsilon_1^{i+2}),\\
&\widetilde{\varphi}(\alpha_2^{i+2})=(\varepsilon_1^i+\varepsilon_1^{i+1}+\varepsilon_1^{i+2})+(-\alpha_1^i-\alpha_2^i-\alpha_1^{i+1}-\alpha_2^{i+1}-\alpha_1^{i+2}-\alpha_2^{i+2}).
\end{aligned}
\]
Note that $\varepsilon_1^i+\varepsilon_1^{i+1}+\varepsilon_1^{i+2} \in L_B(C)$. Since $(\alpha_l, \rho_{\Delta})=1$ for $1 \le l \le 2$ and $\rho_{\Delta}$ as in Subsection \ref{Def of Const A and B}, we have 
\[\widetilde{\varphi}(\alpha_l^{i+j})-(\varepsilon_1^i+\varepsilon_1^{i+1}+\varepsilon_1^{i+2}) \in L_B(C)\]
for $0 \le j \le 2$ and $1 \le l \le 2.$ Hence, we have $\widetilde{\varphi}(\alpha_l^i) \in L_B(C)$  for $ 1 \le i \le 3m$ and $1 \le l \le 2$.
\end{proof}

\begin{lemma}\label{const B  p=3 cal2}
Let $K\subset K_3^m$ be a self-orthogonal code.  Set $C_0 = C_B(K)$.
Then we have
$\alpha_1^{i+j}+2\alpha_2^{i+2}, \alpha_2^{i+j}+2\alpha_2^{i+2} \in \widetilde{\varphi}(L_A(C_0))$
for $i \in \{3k+1 \mid 0 \le k \le m-1\}$ and $0 \le j \le 2$.
\end{lemma}
\begin{proof}
Let $i \in \{3k+1 \mid 0 \le k \le m-1\}$. The following are elements of $\widetilde{\varphi}(L_A(C_0))$: 
\[
\begin{aligned}
&\widetilde{\varphi}(\alpha_1^i)-\widetilde{\varphi}(\alpha_1^{i+2})=-\alpha_1^i-\alpha_2^{i}-\alpha_1^{i+1},\\ 
&\widetilde{\varphi}(\alpha_2^i)-\widetilde{\varphi}(\alpha_1^{i+2})=-\alpha_1^i-\alpha_1^{i+2}-\alpha_2^{i+2},\\
&\widetilde{\varphi}(\alpha_1^{i+1})-\widetilde{\varphi}(\alpha_1^{i+2})=-\alpha_1^{i+1}-\alpha_1^{i+2}-\alpha_2^{i+2},\\ 
&\widetilde{\varphi}(\alpha_2^{i+1})-\widetilde{\varphi}(\alpha_1^{i+2})=-\alpha_1^i-\alpha_2^{i}-\alpha_1^{i+2},\\
&3\widetilde{\varphi}(\alpha_1^{i+2})=2\alpha_1^i+\alpha_2^i+2\alpha_1^{i+1}+\alpha_2^{i+1}+2\alpha_1^{i+2}+\alpha_2^{i+2},\\ 
&\widetilde{\varphi}(\alpha_2^{i+2})-\widetilde{\varphi}(\alpha_1^{i+2})=-\alpha_1^i-\alpha_2^{i}-\alpha_1^{i+1}-\alpha_2^{i+1}-\alpha_1^{i+2}-\alpha_2^{i+2}. 
\end{aligned}
\]
Then the following integer matrix is obtained from the above coefficients:
\[
\begin{pmatrix}
-1 & -1 & -1 & 0 & 0 & 0  \\
-1 & 0 & 0 & 0 & -1 & -1  \\
0 & 0 & -1 & 0 & -1 & -1  \\
-1 & -1 & 0 & 0 & -1 & 0  \\
2 & 1 & 2 & 1 & 2 & 1 \\
-1 & -1 & -1 & -1 & -1 & -1 
\end{pmatrix}.
\]
By computing the echelon form of this matrix over the integers, we have 
\[
\begin{pmatrix}
1 & 0 & 0 & 0 & 0 & 2  \\
0 & 1 & 0 & 0 & 0 & 2  \\
0 & 0 & 1 & 0 & 0 & 2  \\
0 & 0 & 0 & 1 & 0 & 2  \\
0 & 0 & 0 & 0 & 1 & 2  \\
0 & 0 & 0 & 0 & 0 & 3 
\end{pmatrix}.
\]
This means that
$\alpha_1^{i+j}+2\alpha_2^{i+2}, \alpha_2^{i+j}+2\alpha_2^{i+2} \in \widetilde{\varphi}(L_A(C_0))$
for $i \in \{3k+1 \mid 0 \le k \le m-1\}$ and $0 \le j \le 2$.
\end{proof}

\begin{lemma}\label{iso between Const A and B in p=3}
Let $K \subset K_3^m$ be a self-orthogonal code.  Set $C=C_A(K)$ and $C_0=C_B(K)$. Then we have $L_A(C_0) \cong L_B(C)$ as lattices.
\end{lemma}
\begin{proof}

We first prove that $\widetilde{\varphi}(L_A(C_0)) \subset L_B(C)$. By Lemma \ref{generating set: ConstA and B} (1), it suffices to show that $\widetilde{\varphi}(\alpha_l^i) \in L_B(C)$  for $ 1 \le i \le 3m$, $1 \le l \le 2$ and $\widetilde{\varphi}(\lambda_c) \in L_B(C)$ for any $c \in K \cup \{u_i+2u_m \mid 1\le i \le m\}$, where $\lambda_c$ is as in (\ref{def of lambda}) and $u_i$ is as in (\ref{def of u}). 
By Lemma \ref{Const B p=3 constA}, we have  $\widetilde{\varphi}(\alpha_l^i) \in L_B(C)$  for $ 1 \le i \le 3m$ and $1 \le l \le 2$.  
Let $i \in \{3k+1 \mid 0 \le k \le m-1\}$.
Note that $\lambda_1=\varepsilon_1$ and $\lambda_2=2\varepsilon_1-\alpha_1.$
Since
\[
\begin{aligned}
\widetilde{\varphi}(\varepsilon_1^i+\varepsilon_1^{i+1}+\varepsilon_1^{i+2})
 &= (2\varepsilon_1^i-\alpha_1^i+
     2\varepsilon_1^{i+1}-\alpha_1^{i+1}+
     2\varepsilon_1^{i+2}-\alpha_1^{i+2})+(-\alpha_2^i-\alpha_2^{i+2}),\\
\end{aligned}
\]
we have  $\widetilde{\varphi}(\lambda_c) \in L_B(C)$ for any $c \in  \{u_i+2u_m \mid 1\le i \le m\}$. 
By direct calculations, we have the following:
\begin{align}
\label{Const B p=3 gen 1} &\widetilde{\varphi}(\varepsilon_1^i+2\varepsilon_1^{i+1}-\alpha_1^{i+1})= (\varepsilon_1^i+2\varepsilon_1^{i+1}-\alpha_1^{i+1})+(-\alpha_1^i-\alpha_2^i),\\
\label{Const B p=3 gen 2} &\widetilde{\varphi}(2\varepsilon_1^i-\alpha_1^i+\varepsilon_1^{i+1})=(2\varepsilon_1^{i+1}-\alpha_1^{i+1}+\varepsilon_1^{i+2})+(-\alpha_1^{i+2}-\alpha_2^{i+2}). 
\end{align}
Since $K$ is self-orthogonal,  we see that $\widetilde{\varphi}(\lambda_c) \in L_B(C)$ for any $c \in K$. 

Next, we prove that $L_B(C) \subset \widetilde{\varphi}(L_A(C_0))$. By Lemma \ref{generating set: ConstA and B} (2), it suffices to show that $L_B(\{\boldsymbol{0}\}) \subset  \widetilde{\varphi}(L_A(C_0)) $ and $\lambda_c \in \widetilde{\varphi}(L_A(C_0))$ for any $c \in K \cup \{u_i \mid 1 \le i \le m\}$.  
We verify $L_B(\{\boldsymbol{0}\}) \subset  \widetilde{\varphi}(L_A(C_0))$. Since $L_B(\{\boldsymbol{0}\})=\langle \alpha_1^k+2\alpha_2^{3m}, \alpha_2^k+2\alpha_2^{3m} \mid 1 \le k \le 3m \rangle$, it suffices to show that  $\alpha_1^k+2\alpha_2^{3m}, \alpha_2^k+2\alpha_2^{3m} \in \widetilde{\varphi}(L_A(C_0))$ for $1 \le k \le 3m$. 
By Lemma \ref{const B  p=3 cal2}, 
we see that 
\begin{equation}\label{equ const B p=3}
\alpha_1^{i+j}+2\alpha_2^{i+2}, \alpha_2^{i+j}+2\alpha_2^{i+2} \in \widetilde{\varphi}(L_A(C_0))
\end{equation}
for $0 \le j \le 2$.
Since 
\[\widetilde{\varphi}((\varepsilon_1^i+\varepsilon_1^{i+1}+\varepsilon_1^{i+2})-2\alpha_1^{i+2})=-\alpha_1^i-\alpha_2^i-\alpha_1^{i+1}-\alpha_1^{i+2}-\alpha_2^{i+2},\] 
we have $x_i+2x_{m-1} \in \widetilde{\varphi}(L_A(C_0)),$
 where $x_i=-\alpha_1^i-\alpha_2^i-\alpha_1^{i+1}-\alpha_1^{i+2}-\alpha_2^{i+2}$. Combining this with (\ref{equ const B p=3}), we have  $\alpha_1^k+2\alpha_2^{3m}, \alpha_2^k+2\alpha_2^{3m} \in \widetilde{\varphi}(L_A(C_0))$ for $1 \le k \le 3m$. Finally, we verify that $\lambda_c \in \widetilde{\varphi}(L_A(C_0))$ for any $c \in K \cup \{u_i \mid 1 \le i \le m\}$.  
Since  
$\widetilde{\varphi}(\alpha_1^{i+2})=\varepsilon_1^i+\varepsilon_1^{i+1}+\varepsilon_1^{i+2}$, 
we see that  $\lambda_c \in \widetilde{\varphi}(L_A(C_0))$ for any $c \in  \{u_i \mid 1 \le i \le m\}$. By (\ref{Const B p=3 gen 1}) and (\ref{Const B p=3 gen 2}), we have $\lambda_c \in \widetilde{\varphi}(L_A(C_0))$ for any $c \in K$.
\end{proof}

By \cite{HM1}, we obtain a classification of the maximal self-orthogonal codes over $\mathbb{F}_3$ of length at most $9$. 
By using MAGMA with this classification, we have the following lemma: 

\begin{lemma}\label{Const B p=3 length at most 9}
Let $k$ be an integer such that $1 \le k \le 9$ and $C,D$  self-orthogonal codes of length $k$ over $\mathbb{F}_3$.
If $L_B(C) \cong L_B(D)$ as lattices, then we have $C \cong D$ as codes.
\end{lemma}

\subsection{Preparations  for the case $p =5$}
Let $d_5$ denotes the subspace  $\langle (1, 2) \rangle \subset \mathbb{F}_5^2$.  Define a map $\bar{\cdot}: d_5 \rightarrow \mathbb{F}_5$ by $a (1, 2) \mapsto a$. Let $m$ be a positive integer and set $(d_5^{m})_{0} = \{ (x_1, \ldots, x_m) \in d_5^{m} \mid \bar{x}_1+\cdots+\bar{x}_m = 0\}$. The above codes $d_5^m$ and $(d_5^{m})_{0}$ are analogues of codes constructed by Construction A and B from a Kleinian code in \cite[Subsection 1.5]{Shima}. We identify the vector space $d_5^{m}$ with subsets of $\mathbb{F}_5^{2m}$. Note that the code $d_5^m$ is self-orthogonal.
For $1 \le i \le m$, we define 
\begin{equation}\label{def of v}
v_i = (v_{i,1}, v_{i,2}, \ldots, v_{i,2m}) \in \mathbb{F}_5^{2m} 
\end{equation}
by $v_{i,2i-1}=1, v_{i,2i}=2$ and $v_{i,j}=0$ otherwise. Note that $d_5^m=\langle v_i \mid 1 \le i \le m\rangle$ and $(d_5^m)_0=\langle v_i+4v_m \mid 1 \le i \le m \rangle$.

The following proposition gives conditions for a code over $\mathbb{F}_5$ to be realized by $d_5^m$ and  $(d_5^{m})_{0}$:
\begin{proposition}\label{conditions for code over F5}
Let $C$ be a self-orthogonal code of length $2m$ over $\mathbb{F}_5$. Then the following hold:
\begin{enumerate}
\item If $d_5^{m} \subset C$, then we have $C = d_5^m$.
\item \label{conditions for code: Const B} Suppose that $(d_5^m)_0 \subset C$ and $v_1 \in C^{\perp} \setminus C$. Then we have $C = (d_5^m)_0$ and $C+\mathbb{F}_5v_1 = d_5^m$. 
\end{enumerate}
\end{proposition} 

\begin{proof}
(1) It suffices to prove that $C \subset d_5^m.$ Let $x=(x_1,x_2, \ldots, x_m) \in C$, where $x_i \in \mathbb{F}_5^2$ for each $i$. 
Since $(x_i, (1,2))=0$ for each $i$, we have $x_i \in \langle (1,2) \rangle^{\perp}=\langle (1,2)\rangle$ for each $i$. Hence, we have $x \in d_5^m$. 

\noindent (2) We first prove that $C=(d_5^m)_0$. It suffices to prove that $C \subset (d_5^m)_0$. Let $x=(x_1,x_2, \ldots, x_m) \in C$, where $x_i \in \mathbb{F}_5^2$ for each $i$. Since $v_1 \in C^{\perp} \setminus C$ and $4v_1+v_i \in (d_5^m)_0$ for each $i$, we have $v_i \in  C^{\perp} \setminus C$ for each $i$. Hence, we see that $x \in d_5^m$. Let $x=a_1v_1+ a_2v_2+ \cdots+a_m v_m$, where $a_i \in \mathbb{F}_5$ for each $i$. Since $v_i+4v_m \in C$ and $v_i \in C^{\perp} \setminus C$ for each $i$, we have $ x \in (d_5^m)_0$. Hence, we see that  $C = (d_5^m)_0$ and $C+\mathbb{F}_5v_1 = d_5^m$. 
\end{proof}

By using the above proposition,  we have another criterion for a code over $\mathbb{F}_5$ to be realized by  $d_5^m$ and  $(d_5^{m})_{0}$.
\begin{lemma}\label{another criterion for Const B in p=5}
Let $C$ be a self-orthogonal code of length $k$ over $\mathbb{F}_5$. Suppose that $C$ satisfies the following condition:
\begin{equation}\label{condition for another criterion} 
v_1 \in C^{\perp}\setminus C\ \text{and}\  k/2+\#C(2) \le \#(v_1+C)(2). 
\end{equation}
Then we have $ 2 \mid k $ and we have $C \cong (d_5^{k/2})_{0}$ and $C+\mathbb{F}_5v_1 \cong d_5^{k/2} $ as codes.
\end{lemma}

\begin{proof}
This proof is similar to the proof of Lemma \ref{another criterion for Const B in p=3}. We assume that $1 \le r < k/2$ and $\{v_i \mid 1 \le i \le r\} \subset v_1+C$. Let $x=(x_1,\ldots, x_r,y) \in (v_1+C)(2)$, where $x_i \in \mathbb{F}_5^2$ for each $i$ and $y \in \mathbb{F}_5^{k-2r}$. Since $v_i-x \in C$ and $v_i \in C^{\perp}$ for $1 \le i \le r$, we have $(v_i-x,v_i)=0$ for $1 \le i \le r$. Since $(x_i, (1,2))=0$ for $1 \le i \le r$, we see that  $x_i \in \langle (1,2) \rangle$ for $1 \le i \le r$. Hence, if $x_i \neq 0$ for some $ 1 \le i \le r$, then we obtain $x =v_i$. 

Let $V_r =\{v_i \mid 1 \le i \le r\}$.  Then we have the following: 
\begin{align*}
\#\{(0,\ldots,0, y) \in (v_1+C)(2) \mid y \in \mathbb{F}_5^{k-2r}\}
&=\#(v_1+C)(2)-\#V_r \\
&>\#(v_1+C)(2)-k/2-\#C(2) \\ 
&\ge 0.
\end{align*}
Hence, there exists a self-orthogonal code $C'\cong C$ which satisfies  (\ref{condition for another criterion}) and  $\{v_i \mid 1 \le i \le r+1\} \subset v_1 +C'$. This implies $2 \mid k$. Moreover, by Proposition \ref{conditions for code over F5}, we have  $C \cong (d_5^{k/2})_{0}$ and $C+\mathbb{F}_5v_1 \cong d_5^{k/2} $ as codes.  
\end{proof}

Let us establish an isomorphism between the lattices $L_A((d_5^m)_0)$ and $L_B(d_5^m)$. 
Let $\alpha_i\ (1\le i \le 4)$ and $\varepsilon_j\ (1 \le j \le 5)$ be as in Subsection \ref{Def of Const A and B}. 
Note that $\varepsilon_1=(1/5)(4\alpha_1+3\alpha_2+2\alpha_3+\alpha_4)$ and $\varepsilon_i =\varepsilon_1-\alpha_1-\cdots-\alpha_{i-1}$ for $2 \le i \le 5$.
We define a linear map $\psi: \bigoplus_{i=1}^2 \langle \alpha_i \mid 1 \le i \le 4 \rangle_{\mathbb{R}} \rightarrow  \bigoplus_{i=1}^2 \langle \alpha_i \mid 1 \le i \le 4 \rangle_{\mathbb{R}}$ as follows: 
\[
\begin{aligned}
&\psi((\alpha_1,0))=(\varepsilon_1,-(\varepsilon_1+\varepsilon_2+\varepsilon_4)), 
\psi((\alpha_2,0))=(\varepsilon_3, \varepsilon_2+\varepsilon_4),\\
&\psi((\alpha_3,0))=(\varepsilon_5, \varepsilon_1+\varepsilon_3),
\psi((\alpha_4,0))=(\varepsilon_2, -(\varepsilon_1+\varepsilon_3+\varepsilon_4)),\\
&\psi((0,\alpha_1))=(\varepsilon_1,\varepsilon_1+\varepsilon_2),
\psi((0,\alpha_2))=(\varepsilon_2,\varepsilon_3+\varepsilon_4),\\
&\psi((0,\alpha_3))=(\varepsilon_3,-(\varepsilon_2+\varepsilon_3+\varepsilon_4)),
\psi((0,\alpha_4))=(\varepsilon_4,\varepsilon_2+\varepsilon_3).
\end{aligned}
\]
By using (\ref{varepsilon:inner product}), we can directly check that $\psi$ preserves the inner product $(\cdot,\cdot)$.
 Let $m$ be a positive integer.  We define a map $\widetilde{\psi}: \bigoplus_{j=1}^m( \bigoplus_{i=1}^2 \langle \alpha_i \mid 1 \le i \le 4 \rangle_{\mathbb{R}} ) \rightarrow \bigoplus_{j=1}^m(\bigoplus_{i=1}^2 \langle \alpha_i \mid 1 \le i \le 4 \rangle_{\mathbb{R}} )$ by 
\[(x_1, \ldots, x_m) \mapsto (\psi(x_1), \ldots, \psi(x_m)).\] 
Let $\alpha_j^{i}$ and $\varepsilon_j^{i}$ denotes the element $\alpha_j$ and $\varepsilon_j$ in the $i$-th coordinate, respectively. 

\begin{lemma}\label{Const B p=5 constA} 
Let $C = d_5^m$.
Then we have $\widetilde{\psi}(\alpha_l^i) \in L_B(C)$  for $ 1 \le i \le 2m$ and $1 \le l \le 4$.
\end{lemma}

\begin{proof}
Let $i \in \{2k+1 \mid 0 \le k \le m-1\}$. 
By the definition of $\widetilde{\psi}$ and direct calculations, we have 
\[
\begin{aligned}
&\widetilde{\psi}(\alpha_1^i)=(\varepsilon_1^i+2\varepsilon_1^{i+1}-\alpha_1^{i+1})+(-\alpha_1^{i+1}-2\alpha_2^{i+1}-\alpha_3^{i+1}-\alpha_4^{i+1}),\\ 
&\widetilde{\psi}(\alpha_2^i)=(\varepsilon_1^i+2\varepsilon_1^{i+1}-\alpha_1^{i+1})+(-\alpha_1^i-\alpha_2^i-\alpha_1^{i+1}-\alpha_2^{i+1}-\alpha_3^{i+1}),\\
&\widetilde{\psi}(\alpha_3^i)=(\varepsilon_1^i+2\varepsilon_1^{i+1}-\alpha_1^{i+1})+(-\alpha_1^i-\alpha_2^i-\alpha_3^i-\alpha_4^i-\alpha_2^{i+1}),\\
&\widetilde{\psi}(\alpha_4^i)=(\varepsilon_1^i+2\varepsilon_1^{i+1}-\alpha_1^{i+1})+(-\alpha_1^i-\alpha_1^{i+1}-\alpha_2^{i+1}-\alpha_3^{i+1}-\alpha_4^{i+1}),\\
&\widetilde{\psi}(\alpha_1^{i+1})=\varepsilon_1^i+2\varepsilon_1^{i+1}-\alpha_1^{i+1},\\
&\widetilde{\psi}(\alpha_2^{i+1})=(\varepsilon_1^i+2\varepsilon_1^{i+1}-\alpha_1^{i+1})+(-\alpha_1^i-\alpha_1^{i+1}-2\alpha_2^{i+1}-\alpha_3^{i+1}),\\
&\widetilde{\psi}(\alpha_3^{i+1})=(\varepsilon_1^i+2\varepsilon_1^{i+1}-\alpha_1^{i+1})+(-\alpha_1^i-\alpha_2^i-\alpha_2^{i+1}-\alpha_3^{i+1}-\alpha_4^{i+1}),\\
&\widetilde{\psi}(\alpha_4^{i+1})=(\varepsilon_1^i+2\varepsilon_1^{i+1}-\alpha_1^{i+1})+(-\alpha_1^i-\alpha_2^i-\alpha_3^i-\alpha_1^{i+1}-\alpha_2^{i+1}).
\end{aligned}
\]
Note that $\varepsilon_1^i+2\varepsilon_1^{i+1}-\alpha_1^{i+1} \in L_B(C)$. Since $(\alpha_l, \rho_{\Delta})=1$ for $1 \le l \le 4$ and $\rho_{\Delta}$ as in Subsection \ref{Def of Const A and B}, we have 
\[\widetilde{\psi}(\alpha_l^{i+j})-(\varepsilon_1^i+2\varepsilon_1^{i+1}-\alpha_1^{i+1}) \in L_B(C)\]
for $0 \le j \le 1$ and $1 \le l \le 4.$ Hence, we have $\widetilde{\psi}(\alpha_l^i) \in L_B(C)$  for $ 1 \le i \le 2m$ and $1 \le l \le 4$.
\end{proof}

\begin{lemma}\label{const B  p=5 cal2}
Let $C_0=(d_5^m)_0$. 
Then we have
$\alpha_l^{i+j}+4\alpha_4^{i+1} \in \widetilde{\psi}(L_A(C_0))$
for $i \in \{2k+1 \mid 0 \le k \le m-1\}$, $0 \le j \le 1$, and $1 \le l \le 4$.
\end{lemma}
\begin{proof}
Let $i \in \{2k+1 \mid 0 \le k \le m-1\}$. The following are elements of $\widetilde{\psi}(L_A(C_0))$: 
\[
\begin{aligned}
&\widetilde{\psi}(\alpha_1^i)-\widetilde{\psi}(\alpha_1^{i+1})=-\alpha_1^{i+1}-2\alpha_1^{i+1}-\alpha_3^{i+1}-\alpha_4^{i+1},\\
&\widetilde{\psi}(\alpha_2^i)-\widetilde{\psi}(\alpha_1^{i+1})=-\alpha_1^i-\alpha_2^{i}-\alpha_1^{i+1}-\alpha_2^{i+1}-\alpha_3^{i+1},\\
&\widetilde{\psi}(\alpha_3^i)-\widetilde{\psi}(\alpha_1^{i+1})=-\alpha_1^i-\alpha_2^{i}-\alpha_3^{i}-\alpha_4^i-\alpha_2^{i+1},\\
&\widetilde{\psi}(\alpha_4^{i})-\widetilde{\psi}(\alpha_1^{i+1})=-\alpha_1^i-\alpha_1^{i+1}-\alpha_2^{i+1}-\alpha_3^{i+1}-\alpha_4^{i+1},\\
&5\widetilde{\psi}(\alpha_1^{i+1})=4\alpha_1^i+3\alpha_2^i+2\alpha_3^i+\alpha_4^i+3\alpha_1^{i+1}+6\alpha_2^{i+1}+4\alpha_3^{i+1}+2\alpha_4^{i+1},\\
&\widetilde{\psi}(\alpha_2^{i+1})-\widetilde{\psi}(\alpha_1^{i+1})=-\alpha_1^i-\alpha_1^{i+1}-2\alpha_2^{i+1}-\alpha_3^{i+1},\\
&\widetilde{\psi}(\alpha_3^{i+1})-\widetilde{\psi}(\alpha_1^{i+1})=-\alpha_1^i-\alpha_2^{i}-\alpha_2^{i+1}-\alpha_3^{i+1}-\alpha_4^{i+1},\\
&\widetilde{\psi}(\alpha_4^{i+1})-\widetilde{\psi}(\alpha_1^{i+1})=-\alpha_1^i-\alpha_2^{i}-\alpha_3^i-\alpha_1^{i+1}-\alpha_2^{i+1}.\\
\end{aligned}
\]
Then the following integer matrix is obtained from the above coefficients:
\[
\begin{pmatrix}
0 & 0 & 0 & 0 & -1 & -2 & -1 & -1 \\
-1 & -1 & 0 & 0 & -1 & -1 & -1 & 0 \\
-1 & -1 & -1 & -1 & 0 & -1 & 0 & 0 \\
-1 & 0 & 0 & 0 & -1 & -1 & -1 & -1 \\
4 & 3 & 2 & 1 & 3 & 6 & 4 & 2 \\
-1 & 0 & 0 & 0 & -1 & -2 & -1 & 0 \\
-1 & -1 & 0 & 0 & 0 & -1 & -1 & -1 \\
-1 & -1 & -1 & 0 & -1 & -1 & 0 & 0
\end{pmatrix}.
\]
By computing the echelon form of this matrix over the integers, we have 
\[
\begin{pmatrix}
1 & 0 & 0 & 0 & 0 & 0 & 0 & 4 \\
0 & 1 & 0 & 0 & 0 & 0 & 0 & 4 \\
0 & 0 & 1 & 0 & 0 & 0 & 0 & 4 \\
0 & 0 & 0 & 1 & 0 & 0 & 0 & 4 \\
0 & 0 & 0 & 0 & 1 & 0 & 0 & 4 \\
0 & 0 & 0 & 0 & 0 & 1 & 0 & 4 \\
0 & 0 & 0 & 0 & 0 & 0 & 1 & 4 \\
0 & 0 & 0 & 0 & 0 & 0 & 0 & 5
\end{pmatrix}.
\]
This means that
$\alpha_l^{i+j}+4\alpha_4^{i+1} \in \widetilde{\psi}(L_A(C_0))$
for $0 \le j \le 1$ and $1 \le l \le 4$.
\end{proof}

\begin{lemma}\label{iso between Const A and B in p=5}
Let $C=d_5^m$ and $C_0=(d_5^m)_0$.  Then we have $L_A(C_0) \cong L_B(C)$ as lattices.
\end{lemma}
\begin{proof}
We first prove that $\widetilde{\psi}(L_A(C_0)) \subset L_B(C)$. By Lemma \ref{generating set: ConstA and B} (1), it suffices to show that $\widetilde{\psi}(\alpha_j^i) \in L_B(C)$  for $ 1 \le i \le 2m$, $1 \le j \le 4$ and $\widetilde{\psi}(\lambda_c) \in L_B(C)$ for any $c \in \{v_i+4v_m \mid 1 \le i \le m\}$, where  $\lambda_c$ is as in 
(\ref{def of lambda}) and $v_i$ is as in (\ref{def of v}). 
By Lemma \ref{Const B p=5 constA}, we have  $\widetilde{\psi}(\alpha_l^i) \in L_B(C)$  for $ 1 \le i \le 2m$ and $1 \le l \le 4$. 
Let $i \in \{2k+1 \mid 0 \le k \le m-1\}$.
Note that $\lambda_1 =\varepsilon_1$, $\lambda_2=2\varepsilon_1-\alpha_1$, $\lambda_3=3\varepsilon_1-2\alpha_1-\alpha_2$. Since
\[\widetilde{\psi}(\varepsilon_1^i+2\varepsilon_1^{i+1}-\alpha_1^{i+1})=(3\varepsilon_1^i-2\alpha_1^i-\alpha_2^i+\varepsilon_1^{i+1})+(-\alpha_1^{i+1}-\alpha_2^{i+1}),\]
we see that $\widetilde{\psi}(\lambda_c) \in L_B(C)$ for any $c \in \{v_i+4v_m \mid 1 \le i \le m\}$. 
 
 Next, we prove that $L_B(C) \subset \widetilde{\psi}(L_A(C_0))$. By Lemma \ref{generating set: ConstA and B} (2), it suffices to show that $L_B(\{\boldsymbol{0}\}) \subset  \widetilde{\psi}(L_A(C_0))$ and $\lambda_c \in \widetilde{\psi}(L_A(C_0))$ for any $c \in \{v_i \mid 1\le i \le m\}$.  
We verify $L_B(\{\boldsymbol{0}\}) \subset  \widetilde{\psi}(L_A(C_0))$. Since $L_B(\{\boldsymbol{0}\})=\langle \alpha_l^k+4\alpha_4^{2m} \mid 1 \le k \le 2m, 1\le l \le 4 \rangle$, it suffices to show that  $\alpha_l^k+4\alpha_4^{2m} \in \widetilde{\psi}(L_A(C_0))$ for $1 \le k \le 2m$ and $1\le l \le 4$. 
By Lemma \ref{const B  p=5 cal2}, 
we see that
\begin{equation}\label{equ const B p=5}
\alpha_l^{i+j}+4\alpha_4^{i+1} \in \widetilde{\psi}(L_A(C_0))
\end{equation}
for $0 \le j \le 1$ and $1 \le l \le 4$.
Since
\[\widetilde{\psi}((\varepsilon_1^i+2\varepsilon_1^{i+1}-\alpha_1^{i+1})-3\alpha_1^{i+1})=-2\alpha_1^i-\alpha_2^i-2\alpha_1^{i+1}-4\alpha_2^{i+1}-2\alpha_3^{i+1}-\alpha_4^{i+1},\]
we have
$2x_i+3x_{m-1} \in \widetilde{\psi}(L_A(C_0))$,
where $x_i=-2\alpha_1^i-\alpha_2^i-2\alpha_1^{i+1}-4\alpha_2^{i+1}-2\alpha_3^{i+1}-\alpha_4^{i+1}.$
Combining this with (\ref{equ const B p=5}), we have
$\alpha_l^k+4\alpha_4^{2m} \in \widetilde{\psi}(L_A(C_0))$ for $1 \le k \le 2m$ and $1\le l \le 4$. 
Moreover, since  $L_B(\{\boldsymbol{0}\}) \subset  \widetilde{\psi}(L_A(C_0))$ and 
$\widetilde{\psi}(\alpha_1^{i+1})=\varepsilon_1^i+2\varepsilon_1^{i+1}-\alpha_1^{i+1} \in \widetilde{\psi}(L_A(C_0)),$
we also see that $\lambda_c \in \widetilde{\psi}(L_A(C_0))$ for any $c \in \{v_i \mid 1 \le i \le m\}$.
\end{proof}

By \cite{HM2}, we obtain a classification of the maximal self-orthogonal codes over $\mathbb{F}_5$ of length at most $5$. 
By using MAGMA with this classification, we have the following lemma: 

\begin{lemma}\label{Const B p=5 length at most 5}
Let $k$ be an integer such that $1 \le k \le 5$ and $C,D$ self-orthogonal codes of length $k$ over $\mathbb{F}_5$.
If $L_B(C) \cong L_B(D)$ as lattices, then we have $C \cong D$ as codes.
\end{lemma}

\subsection{Preparations for the cases $p=7,11$}
In this subsection, we prove some lemmas for self-orthogonal codes over $\mathbb{F}_7$ and $\mathbb{F}_{11}$.
Let $p$ be $7$ or $11$.
\begin{lemma}\label{Const B p=7,11 length at most 3,2}
Let $k$ be an integer %
and $C,D$ self-orthogonal codes of length $k$ over $\mathbb{F}_p$.
We suppose that $L_B(C) \cong L_B(D)$ as lattices. Then the following hold:
\begin{enumerate}
    \item If $p=7$ and $1 \le k \le 3$, then $C \cong D$ as codes.
    \item If $p=11$ and $1 \le k \le 2$, then $C \cong D$ as codes.
\end{enumerate}
\end{lemma}
\begin{proof}
(1) We classify all self-orthogonal codes $C$ of length $k$ over $\mathbb{F}_7$.  
Since $C \subset C^{\perp}$ and $\operatorname{dim}C + \operatorname{dim}C^{\perp}=k$, we have $\operatorname{dim}C=0,1$.
If $\operatorname{dim}C=0$, then we see that $C=\{\boldsymbol{0}\}$. Hence, we consider the case where $\operatorname{dim}C=1$. Then we have $C=\langle (1,2,3) \rangle$, up to isomorphism.
Therefore,  if $L_B(C) \cong L_B(D)$ as lattices, then we have $C \cong D$ as codes. 

\noindent (2)
We classify all self-orthogonal codes $C$ of length $k$ over $\mathbb{F}_{11}$.  
Imitating the above proof, we have $\operatorname{dim}C=0,1$. By direct calculations, we see that there are no self-orthogonal codes with $\operatorname{dim}C=1$. Hence, if $L_B(C) \cong L_B(D)$ as lattices, then we have $C \cong D$ as codes.
\end{proof}

\subsection{Proof of Theorem \ref{Answer of iso prob for Const B}}
In this subsection, we prove Theorem \ref{Answer of iso prob for Const B}.
\begin{proof}[Proof of Theorem \ref{Answer of iso prob for Const B}]
If $C \cong D$ as codes, then we clearly have $L_B(C) \cong L_B(D)$ as lattices. Hence, we prove that if $L_B(C) \cong L_B(D)$ as lattices, then $C \cong D$ as codes.

Let $R=A_{p-1}^k$. Let $\alpha_j^{i}$ and $\varepsilon_j^{i}$ denote the elements $\alpha_j$ and $\varepsilon_j$ in the $i$-th coordinate, respectively. We suppose that $f$ is an isometry from $L_B(C)$ to $L_B(D)$. Then the isometry $f$ naturally induces an isometry from $L_B(C)^{*}$ to $L_B(D)^{*}$. By Lemma \ref{Structures of the dual lattices of Construction B},  
\[L_B(C)^{*}=L_A(C^{\perp})+\mathbb{Z}\chi_{\Delta}\ \text{and}\ L_B(D)^{*}=L_A(D^{\perp})+\mathbb{Z}\chi_{\Delta}.\] 
We consider the two cases  $f(\alpha_1^1) \in L_A(D^{\perp})$ and $f(\alpha_1^1) \notin L_A(D^{\perp})$.

\noindent (I) The case $f(\alpha_1^1) \in L_A(D^{\perp})$

By Lemmas~\ref{minimal norm}~and~\ref{lower bound}, we have $f(\alpha_1^1) \in R(2)$ for $p \ge 7$. 
If $f(\alpha_1^1) \in R(2)$, then we have $f(L_A(C)) = L_A(D)$. 
In this case, Theorem \ref{answer of isom prob for const A} implies $C \cong D$ as codes. Hence, we consider the case $f(\alpha_1^1) \notin R(2)$. Since $f$ is an isometry, we have
\begin{align}
\label{the number of root in the proof0} \#(L_B(C)(2))&=\#(L_B(D)(2)), \\
\label{the number of root in the proof} \#(\alpha_1^1+L_B(C))(2)&=\#(f(\alpha_1^1)+L_B(D))(2).
\end{align}

\noindent (i) The case $p=3$

We may assume that 
$f(\alpha_1^1)=\varepsilon_1^1+\varepsilon_1^2+\varepsilon_1^3$.
By (\ref{the number of root in the proof0}), we have $\#C(3)=\#D(3)$. Hence, by (\ref{the number of root in the proof}), we obtain 
\[3k+9 \cdot\#D(3)=9 \cdot \#(u_1+D)(3),\]
where $u_1=(1,1,1,0,\ldots,0) \in \mathbb{F}_3^k$.
By Lemma \ref{another criterion for Const B in p=3},  we have $3 \mid k$ and there exists a self-orthogonal code $K \subset K_3^{k/3}$ such that
$C_B(K) \cong D\ \text{and}\ C_A(K) \cong D+\mathbb{F}_3u_1$.
Hence
\begin{align*}
L_A(C)
&=L_B(C) \cup (\alpha_1^1+L_B(C)) \cup (2\alpha_1^1+L_B(C)) \\
&\cong L_B(D) \cup (f(\alpha_1^1)+L_B(D)) \cup (2f(\alpha_1^1)+L_B(D)) \\
&= L_B(D+\mathbb{F}_3u_1) \\
&\cong L_B(C_A(K)). \\
\end{align*}
By Lemma \ref{iso between Const A and B in p=3}, since $L_B(C_A(K)) \cong L_A(C_B(K))$, we have $L_A(C) \cong L_A(D).$
By Theorem \ref{answer of isom prob for const A}, we obtain $C \cong D$ as codes.

\noindent (ii) The case $p=5$

This case is similar to the case $p=3$.
We may assume that 
$f(\alpha_1^1)=\varepsilon_1^1+2\varepsilon_1^2-\alpha_1^2$.
By (\ref{the number of root in the proof0}), we have $\#C(2)=\#D(2)$. Hence, by (\ref{the number of root in the proof}), we obtain 
\[5k+10\cdot\#D(2)=10 \cdot \#(v_1+D)(2),\]
where $v_1=(1,2,0,0,\ldots,0) \in \mathbb{F}_5^k$.
By Lemma \ref{another criterion for Const B in p=5}, we have $2 \mid k$ and 
$(d_5^m)_0 \cong D, d_5^m \cong D+\mathbb{F}_5v_1$.
Hence
\begin{align*}
L_A(C)
&= \bigcup_{0 \le i \le 4} (i\alpha_1^1+L_B(C)) \\
&\cong  \bigcup_{0 \le i \le 4} (if(\alpha_1^1)+L_B(D))\\ 
&= L_B(D+\mathbb{F}_5v_1) \\
&\cong L_B(d_5^m). \\
\end{align*}
By Lemma \ref{iso between Const A and B in p=5}, since $L_B(d_5^m) \cong L_A((d_5^m)_0)$, we have $L_A(C) \cong L_A(D).$
By Theorem \ref{answer of isom prob for const A}, we obtain $C \cong D$ as codes.

\noindent (II) The case $f(\alpha_1^1) \notin L_A(D^{\perp})$

Since $f(\alpha_1^1) \notin L_A(D^{\perp})$, there exists an integer $l \in \{1,2,\ldots, p-1\}$ such that $f(\alpha_1^1) \in L_A(D^{\perp}) + l\chi_{\Delta}.$ 
By Lemma \ref{lower bound 2}, we have 
\[N_1(\frac{l}{p}\rho_{\Delta}+A_{p-1}^*) \ge \frac{(p-1)(p+1)}{12p}.\]
Note that $N_1((l/p)\rho_{\Delta}+A_{p-1}^*) \ge  (p+1)/p$ if $p \ge 13$. Hence, we consider the cases $p =3,5,7,11$. By Lemma \ref{lower bound 2}, we have the following:
\[
\begin{aligned}
N_1((l/3)\rho_{\Delta}+A_2^*) &\ge 2/9\ \text{if}\ p=3,\\  
N_1((l/5)\rho_{\Delta}+A_4^*) &\ge 2/5\ \text{if}\ p=5,\\
N_1((l/7)\rho_{\Delta}+A_6^*) &\ge 4/7\ \text{if}\ p=7,\\
N_1((l/11)\rho_{\Delta}+A_{10}^*) &\ge 10/11\ \text{if}\ p=11.
\end{aligned}
\]
Hence, we obtain the following conditions of the length $k$:
\[
\begin{aligned}
&k\le 9\ \text{if}\ p=3,\  
k\le 5\ \text{if}\ p=5,\\
&k\le 3\ \text{if}\ p=7,\
k\le 2\ \text{if}\ p=11.
\end{aligned}
\]
By Lemmas \ref{Const B p=3 length at most 9}, \ref{Const B p=5 length at most 5}, and \ref{Const B p=7,11 length at most 3,2}, we have $C \cong D$ as codes.
In conclusion, we obtain the desired result that  if $L_B(C) \cong L_B(D)$ as lattices, then $C \cong D$ as codes.

%
%
%
%
%
%
%
%
\end{proof}

\paragraph{\textbf{Acknowledgments}}
I would like to express my gratitude to my supervisor Professor Kazuya Kawasetsu for support and encouragement.
I also wish to express my gratitude to Professor Koichi Betsumiya, who is my former supervisor, for useful comments on codes and lattices. 
I thank Professor Hiroki Shimakura for encouragement and useful comments.
I also thank Professor Akihiro Munemasa for giving me a database of self-orthogonal codes over $\mathbb{F}_5$.
This work was supported by JST SPRING, Grant Number JPMJSP2127.

\end{document}